\documentclass[11pt]{amsart}
\usepackage[utf8]{inputenc}
\usepackage[margin=1.25in]{geometry}
\usepackage{amssymb}       
\usepackage{amsmath}
\usepackage{amsthm}
\usepackage{cite}
\usepackage{enumerate}
\usepackage{xcolor}

\newcommand{\R}{{\mathbb R}}

\newcommand{\E}{{\mathbb E}}
\newcommand{\ds}{\displaystyle}

\numberwithin{equation}{section}
\theoremstyle{plain}
\newtheorem{theorem}{Theorem}[section]
 
\newtheorem{lemma}[theorem]{Lemma}

\newtheorem{proposition}[theorem]{Proposition}
\newtheorem*{theorem*}{Theorem}
\newtheorem*{lemma*}{Lemma}

\theoremstyle{definition} 
\newtheorem{defi}[theorem]{Definition}
\newtheorem{hyp}{Hypothesis}

\theoremstyle{remark}

\newtheorem{remark}[theorem]{Remark}

\title[Large deviations for fast transport stochastic RDEs]{Large deviations for fast transport stochastic RDEs with applications to the exit problem}
\author{ Sandra Cerrai and Nicholas Paskal}
\address{Department of Mathematics, University of Maryland, College Park, USA}
\email{cerrai@umd.edu, npaskal@umd.edu}
\date{}

\begin{document}

\def\le{\left}
\def\r{\right}
\def\cost{\mbox{const}}
\def\a{\alpha}
\def\d{\delta}
\def\ph{\varphi}
\def\e{\epsilon}
\def\la{\lambda}
\def\si{\sigma}
\def\La{\Lambda}
\def\B{{\cal B}}
\def\A{{\mathcal A}}
\def\L{{\mathcal L}}
\def\O{{\mathcal O}}
\def\bO{\overline{{\mathcal O}}}
\def\F{{\mathcal F}}
\def\K{{\mathcal K}}
\def\H{{\mathcal H}}
\def\D{{\mathcal D}}
\def\C{{\mathcal C}}
\def\M{{\mathcal M}}
\def\N{{\mathcal N}}
\def\G{{\mathcal G}}
\def\T{{\mathcal T}}
\def\R{{\mathbb R}}
\def\I{{\mathcal I}}

\def\bw{\overline{W}}
\def\phin{\|\varphi\|_{0}}
\def\s0t{\sup_{t \in [0,T]}}
\def\lt{\lim_{t\rightarrow 0}}
\def\iot{\int_{0}^{t}}
\def\ioi{\int_0^{+\infty}}
\def\ds{\displaystyle}
\def\pag{\vfill\eject}
\def\fine{\par\vfill\supereject\end}
\def\acapo{\hfill\break}

\def\beq{\begin{equation}}
\def\eeq{\end{equation}}
\def\barr{\begin{array}}
\def\earr{\end{array}}
\def\vs{\vspace{.1in}   \\}
\def\rd{\reals\,^{d}}
\def\rn{\reals\,^{n}}
\def\rr{\reals\,^{r}}
\def\bD{\overline{{\mathcal D}}}

\newcommand{\red}[1]{\textcolor{red}{#1}}

\maketitle

\section{Introduction}

\begin{abstract}
We study reaction diffusion equations with a deterministic reaction term as well as two random reaction terms, one that acts on the interior of the domain, and another that acts only on the boundary of the domain. We are interested in the regime where the relative sizes of the diffusion and reaction terms are different. Specifically, we consider the case where the diffusion rate is much larger than the rate of reaction, and the deterministic rate of reaction is much larger than either of the random rate of reactions.
\end{abstract}

\bigskip

In this paper we are dealing with equations that describe the evolution of concentrations of various components in a chemical reaction, subject to random influences. We consider the case the concentration is not constant in space in the volume where the reaction takes place, so that the change of concentration due to the spacial transport has to be taken into consideration. Moreover, we assume that random changes in time and space of the rates of reaction occur. All of this leads us  to consider stochastic reaction-diffusion equations in multi-dimentional spatial domains. As a matter of fact, we are considering here the case the noise has an impact not only on the domain of the system, but also on its boundary. As mentioned in \cite{BM2}, classical important examples are heat transfer in a solid in contact with a fluid, chemical reactor theory, colloid and interface chemistry, and analysis of the interactions between air and water on the surfaces of oceans. 
In applications, it is also important to deal with systems where the rates of chemical reactions and the diffusion coefficients have different orders. Here, we are dealing with the regime where the relative size of the diffusion is much larger than the   rates of reaction and the deterministic rate of reaction is much larger than the stochastic rate of reaction.

More precisely, we are considering   the following class of stochastic reaction-diffusion equations,
\begin{equation}\label{eq:main-intro}
\le\{\begin{array}{l}
\ds{ \frac{\partial u_\epsilon}{\partial t}(t,\xi) = \e^{-1}\,\mathcal{A}u_\epsilon(t,\xi) + f(t,\xi,u_\epsilon(t,\xi))  + \alpha(\epsilon)\, g(t,\xi,u_\epsilon(t,\xi)) \frac{\partial w^Q}{\partial t}(t,\xi), \ \  \xi \in \mathcal{O},} \\
\vs
 \ds { \frac{\partial u_\epsilon}{\partial \nu}(t,\xi) = \e\,\beta(\epsilon)\, \sigma(t,\xi) \frac{\partial w^B}{\partial t}(t,\xi), \ \  \xi \in \partial \mathcal{O}, \ \ \ \ \ \  \ds u_\epsilon(0,\xi) = x(\xi), \ \  \xi \in \mathcal{O},} 
 \end{array}\r.
\end{equation}
for  $0 <  \epsilon \ll 1$ and for some positive functions $\a(\e)$ and $\beta(\e)$, both converging to zero, as $\e\to 0$. Here, $\mathcal{O}$ is a bounded domain in $\R^d$, $d\geq 1$,  with a smooth boundary, $\mathcal{A}$ is a  uniformly elliptic second order differential operator, and $\partial/\partial \nu$ is the associated co-normal derivative acting at $\partial \mathcal{O}$. 
The coefficients $f,g:[0,\infty) \times \mathcal{O} \times \R$ satisfy a Lipschitz condition with respect to the third variable. The noises $w^Q$ and $w^B$ are cylindrical Wiener processes valued in $H = L^2(\mathcal{O})$ and $Z = L^2(\partial \mathcal{O})$, respectively, with covariances $Q \in \mathcal{L}^+(H)$ and $B \in \mathcal{L}^+(Z)$. If $d=1$, we can handle space-time white noise, while for $d \geq 2$ we must suppose the noise to be colored in space.

We assume here that the diffusion $X_t$ associated with the operator $\mathcal{A}$, endowed with the co-normal boundary condition, admits a unique invariant measure $\mu$ and   a spectral gap occurs. Namely, there exists some $\gamma>0$ such that for any $h \in L^2(\mathcal{O},\mu)$,
\begin{equation*}
\int_{\mathcal{O}}\le| \mathbb{E}_\xi h(X_t) - \langle h,\mu\rangle\r|^2\,d\mu(\xi) \leq c\,e^{-2\,\gamma t} |h|^2_{L^2(\mathcal{O},\mu)},\ \ \ \ t\geq 0.
\end{equation*}

When the deterministic and stochastic rates of reaction are of order one, the fast diffusion disappears and,   in the limit as $\e\to 0$, the effective dynamic is described by  a ordinary stochastic differential equation. In fact, in \cite{cerrai2011} (see also \cite{BM2}) it was shown that, for every $0<\d<T$ and $p\geq 1$, the solutions $u^x_\e$ to $\eqref{eq:main-intro}$, corresponding to  $\alpha(\e)=\beta(\e)=1$,  converge  in $L^p(\Omega;C([\delta,T];L^2(\mathcal{O},\mu)))$ to the solution of the averaged one-dimensional stochastic differential equation
\begin{equation}\label{eq:intro_avSDE}
du(t) = \bar{F}(t,u(t))dt + \bar{G}(t,u(t)) dw^Q(t) + \bar{\Sigma}(t) dw^B(t), \ \ \ \ 
u(0) = \langle x,\mu\rangle.
\end{equation} 
Here, $\bar{F}$, $\bar{G}$, and $\bar{\Sigma}$ are all obtained by taking suitable spatial averages  of their counterparts, $f, g$ and $\sigma$, with respect to the invariant measure $\mu$. Since the averaging still takes time, convergence in $C([0,T];L^2(\mathcal{O},\mu))$ only occurs if the initial condition $x$ is already constant in space. 

In this paper, we are interested in studying the fast transport approximation described above in the small noise regime (i.e. $\a(\e)\to 0$ and $\beta(\e)\to 0$). In this case, the noisy terms vanish entirely from the limit and the solution to $\eqref{eq:main-intro}$ converges in $L^p(\Omega;C([\delta,T];L^2(\mathcal{O},\mu)))$  to the solution of the ODE
\begin{equation}\label{eq:intro_avODE}
\ds \frac{du}{dt} = \bar{F}(t,u(t)), \ \ \ \ \ u(0) = \langle x,\mu\rangle.
\end{equation}
Thus, we believe it is of interest to study the validity of a large deviation principle for the family $\{u^x_\e\}_{\e>0}$  in the space $C([\delta,T];L^2(\mathcal{O},\mu))$, and, in particular, to understand its interplay with the fast transport limit. It turns out that, depending on the following different scalings between $\a(\e)$ and $\beta(\e)$
\[\lim_{\e\to 0}\,\frac{\beta(\e)}{\a(\e)}=\bar{\rho} \in\,[0,+\infty],
\]
 the action functional and the speed governing the large deviation principle for equation \eqref{eq:main-intro} are precisely the same as those governing the large deviation principle for the SDE 
\[du(t)=\bar{F}(u(t))\,dt+\le(\a(\e)+\beta(\e)\r) \sqrt{\mathcal{H}_{\bar{\rho}}(t,u(t))}\,d\beta_t,\ \ \ \ \ u(0)=\langle x,\mu\rangle,\]
where 
\begin{equation}
\label{f500}
\mathcal{H}_{\bar{\rho}}(t,u)=\frac 1{(1+\bar{\rho})^2}\le[\le|\sqrt{Q}\le[G(t,u)m \r]\r|_H^2 + \bar{\rho}^2\,\le|\delta_0 \sqrt{B}\le[\Sigma(t) N_{\delta_0}^* m\r] \r|_Z^2\r],
\end{equation}
(here $m$ is the density of the invariant measure $\mu$). This means in particular that the fast transport asymptotics for equation \eqref{eq:main-intro}  is consistent with the small noise limit.

\medskip

In the second part of the paper, we study the problem of the exit of the solutions $u_\epsilon^x$ to $\eqref{eq:main-intro}$ from a domain $D$ in the functional space $L^2(\mathcal{O},\mu)$. We consider the case where the limiting equation $\eqref{eq:intro_avODE}$ has an attractive equilibrium at $0$, and we  prove Freidlin-Wentzell type exit time estimates. More precisely, if we define
\begin{equation*}
\tau_\epsilon^x := \inf \{t \geq 0: u_\epsilon^x(t) \in \partial D \},
\end{equation*}
then we show that for any initial condition $x \in D \subset L^2(\mathcal{O},\mu)$,
\begin{equation}\label{eq:intro_exit}
\lim_{\epsilon \to 0} \le(\a(\epsilon)+\beta(\e)\r)^2 \log \mathbb{E} \tau^x_\epsilon = \inf_{y \in \partial D} V(y),
\end{equation}
where $V:L^2(\mathcal{O},\mu) \to \R^+$ is the quasi-potential corresponding to the action functional governing the large deviation principle. If the interior noise is additive, i.e. $g\equiv 1$  in $\eqref{eq:main-intro}$, then the quasi-potential can be written explicitly. Namely,
\begin{equation*}
V(y) = -2\,\mathcal{H}_{\bar{\rho}}^{-1} \int_0^y \bar{F}(r)dr.
\end{equation*}
where $\mathcal{H}_{\bar{\rho}}$ is obtained from \eqref{f500} by setting $G(t,u)=\text{Id}$ and by assuming $\Sigma$ constant in time. For example, when $\mathcal{A}$ is a divergence type operator, we have $m=|\mathcal{O}|^{-1}$ and hence
\[V(y)=-\frac{(1+\bar{\rho})^2}{c_1+c_2\,\bar{\rho}^2}\int_0^y\int_{\mathcal{O}}f(\xi,\si)\,d\xi\,d\si,\]
for some non-negative constants $c_1$ and $c_2$, depending on $Q$ and $B$, and not simultaneously zero.
In the general case of multiplicative noise, we do not have such an explicit representation of the quasi-potential; however, the result $\eqref{eq:intro_exit}$ still holds.
Notice that, as far as we know, this is the first time a result as \eqref{eq:intro_exit} is obtained for an SPDE with multiple scales, where not only the small noise, but also other asymptotics (in this case the fast transport) have to be taken into consideration. 

As known, in order to obtain results like \eqref{eq:intro_exit}, a large deviation principle that is uniform with respect to initial conditions in a bounded set of $L^2(\mathcal{O},\mu)$ is needed. Here, we prove the large deviation principle for the family $\{u^x_\e\}_{\e>0}$ in the space $C([\d,T];L^2(\O,\mu))$ by using the weak convergence approach, as developed for SPDEs in \cite{budhiraja2008}. This method allows to prove a  Laplace principle, which is uniform with respect to initial conditions in a compact set of $H$ and which, as well known, is equivalent to the large deviation principle with the same speed and action functional. Thus, in order to prove \eqref{eq:intro_exit}, we have first to prove that the Laplace principle is uniform with respect to initial conditions on a bounded set of $L^2(\mathcal{O},\mu)$ and then we have to show that a  Laplace principle, that is uniform with respect to initial conditions on a bounded set, implies a large deviation principle that is uniform with respect to initial conditions on the same bounded set. 

In our case the solutions of equation \eqref{eq:main-intro} are converging to a one-dimensional equation, and the problem of proving a uniform Laplace principle for initial data on a bounded set is circumvented by considering the space  $L^2(\mathcal{O})$ endowed with its weak topology.  More delicate is the problem of understanding how the uniform Laplace principle may imply the uniform large deviation principle. To this purpose, recently, in \cite{salins:171207231S}, some conditions have been introduced in order to guarantee, among other things, the equivalence between the uniform Laplace principle and the uniform large deviation principle, with respect to initial conditions in a compact set.  These arguments can be extended in our setting to give uniformity with respect to initial conditions in bounded sets. Thus, our job here is proving that the conditions introduced in \cite{salins:171207231S} are satisfied.

Once we have a large deviation principle that is uniform with respect to initial conditions in a bounded set, we prove $\eqref{eq:intro_exit}$ by adapting the method used in finite dimension  (see Chapter 4, Section 2 of \cite{freidlin1998random} and Chapter 5.7 of \cite{dembo2009large}) to our infinite dimensional setting (see \cite{cf}, \cite{bcf} and \cite{cs} for some previous results in this direction). In our model, several complications arise in obtaining the lower bound of $\mathbb{E} \tau_\epsilon^x$. Actually, when $\epsilon$ is small, equation $\eqref{eq:main-intro}$ behaves like the linear heat equation for $t$ on the order of $\epsilon$. However, for times on the order of $1$, the averaging has already taken place so that the solution is essentially constant in space and evolves according to $\eqref{eq:intro_avODE}$. So to establish any kind of lower bound on the exit time, we require a domain that is both invariant with respect to the semigroup $e^{tA}$ and invariant with respect to trajectories of equation $\eqref{eq:intro_avODE}$.

\section{Notations and preliminaries}

\subsection{Assumptions on the semigroup}
We assume that $\mathcal{O}$ is a bounded domain in $\R^d$, $d\geq 1$,  with a smooth boundary, satisfying the extension and exterior cone properties. We denote $H:= L^2(\mathcal{O})$ and $Z := L^2(\partial \mathcal{O})$, and, for any $\alpha > 0$, we denote $H^\alpha := H^\alpha(\mathcal{O})$ and $Z^\alpha := H^\alpha(\partial \mathcal{O})$. 

We assume that  $\mathcal{A}$ is a second order differential operator of the form
\begin{equation*}
 \mathcal{A}  = \sum_{i,j=1}^d \frac{\partial}{\partial \xi_i}\le( a_{ij}(\xi) \frac{\partial }{\partial \xi_j} \r) + \sum_{i=1}^d b_i(\xi) \frac{\partial }{\partial \xi_i},\ \ \ \ \xi \in\,\mathcal{O}.
\end{equation*}
The matrix  $a(\xi)=[a_{ij}(\xi)]_{i, j}$ is symmetric and all entries $a_{ij}$ are differentiable, with continuous derivatives in $\bar{\mathcal{O}}$. Moreover, there exists some $a_0 > 0$ such that
\begin{equation}
\label{f1}
\inf_{\xi \in \bar{\mathcal{O}}}\,\langle a(\xi) \eta,\eta\rangle  \geq a_0 |\eta|^2, \qquad \eta \in \R^d.
\end{equation}
Finally, the coefficients  $b_i$ are continuous on $\bar{\mathcal{O}}$. 

In what follows, we shall denote by $A$ the realization in $H$ of the differential operator $\mathcal{A}$, endowed with the conormal boundary condition
\[\ds \frac{\partial h }{\partial \nu}(\xi) := \langle a(\xi) \nu(\xi), \nabla h(\xi) \rangle=0,\ \ \ \ \xi \in\,\partial \mathcal{O}.\] 
The operator $A$ generates a strongly continuous analytic semigroup in $H$, which we will denote by $e^{tA}$. Moreover (see  \cite{Lasiecka1980} for a proof)
\begin{equation}\label{eq:sem_domain}
D(A^\alpha) \subseteq H^{2\alpha},\ \ \ \text{ for } \alpha \geq 0,\ \ \ \ \ \ 
D(A^\alpha) = H^{2\alpha} ,\ \ \  \text{ for } 0 \leq \alpha < \frac{3}{4},
\end{equation}

In general, the realization of $\mathcal{A}$ in  $L^p$ spaces under the same boundary conditions will also generate a strongly continuous, analytic semigroup, for $p>1$. It is proved in \cite{davies1990heat} that under the above conditions on $\mathcal{A}$ and $\mathcal{O}$, the semigroup admits an integral kernel $k_t(\xi,\eta)$ that satisfies
\begin{equation}\label{eq:sem_kernel}
0 \leq k_t(\xi,\eta) \leq c\,(t^{-\frac d2}+1), \qquad t > 0, 
\end{equation}

In what follows, we shall assume that  $e^{t A}$ satisfies the following condition.
\begin{hyp}\label{hyp:longterm}
The semigroup $e^{tA}$ admits a unique invariant measure $\mu$, and there exist $\gamma > 0$ and $c > 0$ such that, for any $h \in L^2(\mathcal{O},\mu)$ and $t \geq 0$,
\begin{equation}\label{eq:longterm}
\le|e^{tA}h - \int_\mathcal{O} h(\xi) d\mu(\xi) \r|_{L^2(\mathcal{O},\mu)} \leq c\, e^{-\gamma t} |h|_{L^2(\mathcal{O},\mu)}.
\end{equation}
In what follows, we shall denote 
\[H_\mu:=L^2(\mathcal{O},\mu),\ \ \ \ \ds \langle h,\mu \rangle := \int_\mathcal{O} h(\xi)d \mu(\xi),\ \ h\in H_\mu.\] 
\end{hyp}
\begin{remark} Hypothesis \ref{hyp:longterm} is satisfied  for example if $\mathcal{A}$ is a divergence-type operator. Actually, in this case the Lebesgue measure is invariant under the semigroup $e^{tA}$, so that we  can define \[\mu = |\mathcal{O}|^{-1}\la_d,\]
where $\la_d$ is the Lebesgue measure on $\mathbb{R}^d$.
 Since $A$ is self-adjoint, we can find a complete orthonormal system $\{e_k\}_{k\geq 0}$ in $H$, and an increasing nonnegative sequence  $\{\alpha_k\}_{k\geq 0}$ such that $Ae_k = -\alpha_k e_k$. Clearly, $\alpha_0 = 0$ and $e_0 = |\mathcal{O}|^{-1/2}$, so that   $\langle x,\mu\rangle=\langle x,e_0\rangle_H\,e_0$, for any $x \in H$.
This implies that
\begin{equation*}
|e^{tA}x - \langle x, \mu \rangle |_{H_\mu}^2 = |\mathcal{O}|^{-1} \sum_{i=1}^\infty e^{-2t\alpha_i} \langle x, e_i \rangle_H^2 \leq e^{-2t \alpha_1} |x|_{H_\mu}^2,
\end{equation*}
so that  $\eqref{eq:longterm}$ holds for $\gamma = \alpha_1$.
\begin{flushright}
$\Box$
\end{flushright}

\end{remark}
\begin{remark}\label{rem:embedding} We have the continuous embedding $H \hookrightarrow H_\mu$.
This follows from the invariance of $\mu$ with respect to $e^{tA}$, and from the boundedness of the integral kernel \eqref{eq:sem_kernel}. Actually, for $h \in H$, we have 
\[|h|_{H_\mu}^2    = \int_\mathcal{O} e^{1A} |h|^2(\xi) d\mu(\xi) = \int_\mathcal{O} \int_\mathcal{O}  k_1(\xi,\eta)|h(\eta)|^2 d\eta\, d\mu(\xi) \leq c |h|_H^2.\]
We also note that, due to the invariance of $\mu$,  $e^{tA}$ acts as a contraction in $H_\mu$
\[\begin{array}{l}
\ds{|e^{tA}h|_{H_\mu}^2  \leq  \int_\mathcal{O} e^{t A}\,|h(\xi)|^2\, d\mu(\xi) =  |h|_{H_\mu}^2.}
\end{array}\]
\begin{flushright}
$\Box$
\end{flushright}

\end{remark}

\begin{remark}
\label{rem2.3}
In fact, one can show that the invariant measure $\mu$ is absolutely continuous with respect to the Lebesgue measure on $\mathcal{O}$ and has a nonnegative density $m \in L^\infty(\mathcal{O})$ (for a proof, see \cite{cerrai2011}).
\begin{flushright}
$\Box$
\end{flushright}

\end{remark}

\subsection{Assumptions on the coefficients and noise}
Concerning the coefficients $f,g,$ and $\sigma$, we make the following assumptions.
\begin{hyp}\label{hyp:coeff} \
\begin{enumerate}[(i)]
\item The mappings $f,g:[0,\infty) \times \mathcal{O} \times \R \to \R$ are measurable and Lipschitz continuous in the third variable, uniformly with respect to $(t,\xi) \in [0,T] \times \mathcal{O}$, for any fixed $T > 0$. In addition, for any $T > 0$, $f$ and $g$ satisfy
\[\sup_{0 \leq t \leq T} |f(t,\cdot,0)|_{L^\infty(\mathcal{O})} < +\infty,\ \ \ \ \ 
\sup_{0 \leq t \leq T} |g(t,\cdot,0)|_{L^\infty(\mathcal{O})} < +\infty,
\]
\item The mapping $\sigma:[0,\infty) \times \partial \mathcal{O} \to \R$ is measurable and satisfies for any $T > 0$
\begin{equation*}
\sup_{0 \leq t \leq T} |\sigma(t,\cdot)|_{L^\infty(\partial \mathcal{O})} < +\infty,
\end{equation*}

\end{enumerate}
\end{hyp}
\noindent In what follows, for $h_1,h_2 \in H$ and $\xi \in \mathcal{O}$, we define
\[ F(t,h_1)(\xi) := f(t,\xi,h_1(\xi)), \]
and
\[ [G(t,h_1)h_2](\xi) := g(t,\xi,h_1(\xi))h_2(\xi).\]
The uniform Lipschitz assumptions on $f$ and $g$ in Hypothesis \ref{hyp:coeff} imply that the mappings $F(t,\cdot):H\to H$, $G(t,\cdot):H \to \mathcal{L}(H,L^1(\mathcal{O}))$, and $G(t,\cdot):H \to \mathcal{L}(L^\infty(\mathcal{O}),H)$ are all well-defined and Lipschitz continuous, uniformly with respect to $t \in [0,T]$, for any $T > 0$. 

Next, for $z \in Z$ and $\xi \in \partial \mathcal{O}$, we set
\begin{equation*} [\Sigma(t)z](\xi):=\sigma(t,\xi)z(\xi).
\end{equation*}
Hypothesis \ref{hyp:coeff} implies that $\Sigma(t) \in \mathcal{L}(Z)$ and $\ds \sup_{0 \leq t \leq T} \left\lVert \Sigma(t) \right\rVert_{\mathcal{L}(Z)} < \infty$.

 Concerning the noisy terms, we assume that $w^Q(t)$ and $w^B(t)$ are cylindrical Wiener processes in $H$ and $Z$,  with covariances $Q \in \mathcal{L}^+(H)$ and $B \in \mathcal{L}^+(Z)$, respectively. That is,
\[w^Q(t) = \sum_{k =0}^\infty \sqrt{Q} e_k\beta_k(t),\ \ \ \ \ \ w^B(t) = \sum_{k =0}^\infty \sqrt{B} f_k \tilde{\beta}_k(t),\]
where $\{e_k\}_{k\geq 0}$ is an orthonormal basis of $H$, $\{f_k\}_{k\geq 0}$ is an orthonormal basis of $Z$ and $\{\beta_k(t)\}_{k\geq 0}$ and $\{\tilde{\beta}_k(t)\}_{k\geq 0}$ are sequences of independent real-valued Brownian motions defined on a common stochastic basis $(\Omega, \mathcal{F},\{\mathcal{F}_t\}_{t\geq 0},\mathbb{P})$. 

We assume for simplicity that $\{e_k\}_{k\geq 0}$ diagonalizes $\sqrt{Q}$ with eigenvalues $\{\lambda_k\}_{k\geq 0}$, and $\{f_k\}_{k\geq 0}$ diagonalizes $\sqrt{B}$ with eigenvalues $\{\theta_k\}_{k\geq 0}$. We do not assume that the operators $Q$ and $B$ are trace class, so the sums above do not necessarily converge in $H$ and $Z$, respectively. However, both of the sums converge in larger Hilbert spaces containing $H$ and $Z$, respectively, with Hilbert-Schmidt embeddings. 

We make the following assumption regarding the eigenvalues of $Q$ and $B$.
\begin{hyp}\label{hyp:eigenvalues} If $d\geq 2$, then
 there exist $\rho < 2d/(d-2)$ and $\beta < 2d/(d-1)$ such that
\begin{equation}\label{eq:rhohyp}
\sum_{k \in \mathbb{N}} \lambda_k^\rho |e_k|_\infty^2 =: \kappa_Q < \infty,\ \ \ \ \ \sum_{k \in \mathbb{N}} \theta_k^\beta =:\kappa_B < \infty.
\end{equation}
\end{hyp}

\subsection{Mild solutions}
In the present paper, we are dealing with the following class of equations
\begin{equation}\label{eq:main}
\begin{cases} 
\ds{\frac{\partial u_\epsilon}{\partial t}(t,\xi) = \e^{-1}\,\mathcal{A}u_\epsilon(t,\xi) + f(t,\xi,u_\epsilon(t,\xi))  + \alpha(\epsilon)\, g(t,\xi,u_\epsilon(t,\xi)) \frac{\partial w^Q}{\partial t}(t,\xi), \ \  \xi \in \mathcal{O}, }\\
\vspace{1mm}\\
 \ds {\frac{\partial u_\epsilon}{\partial \nu}(t,\xi) = \e\,\beta(\epsilon)\, \sigma(t,\xi) \frac{\partial w^B}{\partial t}(t,\xi), \ \  \xi \in \partial \mathcal{O}, \ \ \ \ \ \  \ds u_\epsilon(0,\xi) = x(\xi), \ \  \xi \in \mathcal{O}.} \end{cases}
\end{equation}

 Under the above assumptions on the differential operator $\mathcal{A}$ and the domain $\mathcal{O}$, it can be shown (see \cite{lions1972non}), that there exists $\delta_0 \in \R$ such that for any $\delta \geq \delta_0$ and $h \in Z$, the  elliptic boundary value problem
\[(\delta - \mathcal{A})u(\xi) = 0,\ \ \ \ \xi \in \mathcal{O},\ \ \ \ \frac{\partial u}{\partial \nu} = h(\xi),\ \ \ \ \xi \in \partial \mathcal{O},\]
admits a unique solution $u \in H$. We define the Neumann map, $N_\delta:Z \to H$, to be the solution map of this equation, i.e. $N_\delta h := u$. One can show that
\begin{equation}\label{eq:Neuman_bounded}
N_\delta \in \mathcal{L}(Z^\alpha, H^{\alpha + 3/2}).
\end{equation}
Next, we consider the deterministic parabolic problem 
\begin{equation*}
\begin{cases}
\ds {\frac{\partial y}{\partial t}(t,\xi) = \mathcal{A} y(t,\xi),\ \ \ \xi \in \mathcal{O} }\\
\vspace{.1mm}\\
 \ds {
\frac{\partial y}{\partial \nu} = v(t,\xi), \ \ \  \xi \in \partial \mathcal{O},\ \ \ \ \  
y(0,\xi) =0,  \ \ \ \xi \in \mathcal{O}.}
\end{cases}
\end{equation*}
One can show that for smooth $v$ and large enough $\delta$, the solution to this equation is given explicitly by
\begin{equation*}
y(t) = (\delta - A) \int_0^t e^{(t-s)A}N_{\delta}v(s)ds.
\end{equation*} 
This formula can be extended by continuity to provide a notion of mild solution for less regular $v$. In our case, we are interested in the boundary value problem
\begin{equation}\label{eq:Neuman_stoch_prob}
\begin{cases}
\ds{ \frac{\partial y}{\partial t}(t,\xi) = \frac{1}{\epsilon}\mathcal{A} y(t,\xi),\ \ \  \xi \in \mathcal{O}, }\\
 \vspace{.1mm}\\
\ds {\frac{\partial y}{\partial \nu} = \epsilon\,\beta(\e)\,\sigma(t,\xi) \frac{\partial w^B}{\partial t}(t,\xi), \ \ \  \xi \in \partial \mathcal{O}, \ \ \ \ \ y(0,\xi) =0,  \ \ \ \xi \in \mathcal{O},}
\end{cases}
\end{equation}
So, upon taking $\delta = \delta_0/\epsilon$, we say that the process
\begin{equation*}
\beta(\epsilon)\, w_{A,B}^\epsilon(t) := \beta(\epsilon)\,(\delta_0 - A) \int_0^t e^{(t-s)\frac A\e} N_{\delta_0}[\Sigma(s)\,dw^B(s)]
\end{equation*}
is a mild solution to problem $\eqref{eq:Neuman_stoch_prob}$.  (see \cite{dapratozabczyk1996} for details, and see \cite{fw-92}, \cite{SV} and \cite{sowers} for other papers where the same type of equations has been studied). This motivates the following.

\begin{defi}
Let $ p \geq 1$ and $T > 0$. An adapted process $u_\e \in L^p(\Omega;C([0,T];H))$ is called a mild solution to problem $\eqref{eq:main}$ if, for any $ t \in [0,T]$,
\begin{equation*}
u_\e(t) = e^{t\frac A\e}x + \int_0^t e^{(t-s)\frac A\e}F(s,u_\e(s))ds + \a(\epsilon)\, w_{A,Q}^\epsilon(u_\e)(t) + \beta(\epsilon)\,w_{A,B}^\epsilon(t),
\end{equation*}
where, for any $u \in L^p(\Omega;C([0,T];H))$, we define
\begin{equation*}
w_{A,Q}^\epsilon(u)(t):= \int_0^t e^{(t-s)\frac A\e}G(s,u(s))dw^Q(s).
\end{equation*}
\end{defi}

\subsection{Well-posedness and averaging results} In this section, we recall some  important preliminary results from \cite{cerrai2011}. 

\begin{lemma*}[\textbf{3.1 of \cite{cerrai2011}}]
 Assume Hypotheses \ref{hyp:coeff} and \ref{hyp:eigenvalues} hold. Then, for any $\e>0$, $p \geq 1$ and $T > 0$, the process $w_{A,B}^\epsilon$ belongs to $L^p(\Omega;C([0,T];H))$ and satisfies
\begin{equation}\label{eq:Bconv_bound}
\sup_{\epsilon \in (0,1]} \mathbb{E} |w_{A,B}^\epsilon|^p_{C([0,T];H)} < +\infty.
\end{equation}
\end{lemma*}

\begin{lemma*}[\textbf{3.3 of \cite{cerrai2011}}] 
Assume Hypotheses \ref{hyp:coeff} and \ref{hyp:eigenvalues} hold. Then for any $\e>0$, $p \geq 1$ and $T > 0$, the mapping $w_{A,Q}^\epsilon(\cdot)$ maps $L^p(\Omega; C([0,T];H)$ into itself and satisfies
\begin{equation}\label{eq:Qconv_bound}
\sup_{\epsilon \in (0,1]} \mathbb{E} |w_{A,Q}^\epsilon(u)|^p_{C([0,T];H)} \leq c_{T,p} \le(1 + \mathbb{E}\int_0^T |u(s)|_H^p ds \r).
\end{equation}
Moreover, it is Lipschitz continuous and
\begin{equation}\label{eq:Qconv_lip}
\sup_{\epsilon \in (0,1]}|w_{A,Q}^\epsilon(u) - w_{A,Q}^\epsilon(v)|_{L^p(\Omega;C([0,T];H))} \leq L_T |u-v|_{L^p(\Omega;C([0,T];H))},
\end{equation}
for some constant $L_T>0$, independent of $\e \in\,(0,1]$, such that $L_T\to 0$, as $T \to 0$.
\end{lemma*}

\begin{theorem*}[\textbf{3.4 of \cite{cerrai2011}}]
Assume Hypotheses \ref{hyp:coeff} and \ref{hyp:eigenvalues} hold. Then for any $\e>0$, $p \geq 1$ and $T>0$ and for any initial condition $x \in H$, equation $\eqref{eq:main}$ has a unique adapted mild solution $u^x_\epsilon \in L^p(\Omega;C([0,T];H))$, which satisfies 
\begin{equation}\label{eq:apriori_alpha0}
\sup_{\epsilon \in (0,1]} \mathbb{E} |u^x_\epsilon|^p_{C([0,T];H)} \leq c_{T,p}\,(1 + |x|_H^p).
\end{equation}
\end{theorem*}
Next, for any $t \geq 0$ and $h \in H_\mu$, we define
\begin{equation*}
\bar{F}(t,h) := \langle F(t,h), \mu \rangle = \int_\mathcal{O} f(t,\xi,h(\xi)) d\mu(\xi).
\end{equation*}
Moreover, for any $t\geq 0$ and $h_1,h_2 \in H_\mu$, we define
\begin{equation}
\bar{G}(t,h_1) h_2 := \langle G(t,h_1) h_2, \mu \rangle = \int_\mathcal{O} g(t,\xi,h_1(\xi)) h_2(\xi)\, d \mu(\xi),
\end{equation}
 and for any $t\geq 0$ and $z \in Z$, we define
\begin{equation*}
\bar{\Sigma}(t)z = \delta_0\, \langle N_{\delta_0} \Sigma(t)z, \mu \rangle = \delta_0 \int_\mathcal{O} N_{\delta_0} [ \sigma(t,\cdot)z](\xi)\, d \mu(\xi).
\end{equation*}

Hypothesis \ref{hyp:coeff} implies that $\bar{F}(t,\cdot):H_\mu \to \R$ is Lipschitz continuous, uniformly with respect to $t \in [0,T]$. Concerning $\bar{G}$, we observe that for any $h \in\,H_\mu$ and $T>0$
\begin{equation}
\label{eq:Gbar_lip}
\begin{array}{l}
\ds{
|\bar{G}(t,h_1)h - \bar{G}(t,h_2)h|^2  \leq |h|^2_{H_\mu} \int_\mathcal{O} |g(t,\xi,h_1(\xi))-g(t,\xi,h_2(\xi))|^2 d\mu(\xi) }\\
 \vspace{.1mm}\\
\ds{ \leq c |h|_{H_\mu}^2 |h_1-h_2|_{H_\mu}^2,\ \ \ \ h_1, h_2 \in\,H_\mu,\ \ t \in\,[0,T].}
\end{array}\end{equation}
Therefore, $\bar{G}(t,\cdot)h:H_\mu \to \R$ is Lipschitz continuous, uniformly with respect to $t \in [0,T]$ and $h$ in a bounded set of $H_\mu$ (and hence $H$). Finally, the linear functional $\bar{\Sigma}(t):Z \to \R$ is bounded due to $\eqref{eq:Neuman_bounded}$. 
\\ \\
With these notations, we introduce the  equation
\begin{equation}\label{eq:averaged_SDE}
\ds{dv^x(t) = \bar{F}(t,v^x(t))dt + \bar{G}(t,v^x(t)) dw^Q(t) + \bar{\Sigma}(t) dw^B(t),\ \ \ \ v^x(0) = \langle x, \mu \rangle.}
\end{equation} 
\begin{theorem*}[\textbf{4.1 of \cite{cerrai2011}}]  Assume that Hypotheses \ref{hyp:longterm}, \ref{hyp:coeff} and \ref{hyp:eigenvalues} hold, and let $\alpha(\e)=\beta(\e)\equiv 1$. Then, for any $x \in H$, $p \geq 1$, and $0<\d<T$, we have
\begin{equation}\label{eq:SDE_limit}
\lim_{\epsilon \to 0} \mathbb{E} \sup_{\delta \leq t \leq T} |v^x_\epsilon(t)-v^x(t)|_{H_\mu}^p = 0,
\end{equation}
where $v^x_\epsilon$ is the mild solution to $\eqref{eq:main}$ with $\alpha(\e)=\beta(\e)\equiv 1$ and $v^x$ is the solution of equation $\eqref{eq:averaged_SDE}$.
\end{theorem*}

\subsection{Uniform large deviation principle and Laplace principle}
\label{sec:laplace}
In what follows, let $\mathcal{E}$ and $\mathcal{E}_0$ be Polish spaces.  We recall that a function $I: \mathcal{E} \to [0,+\infty]$ is called a {\em good rate function} if the level set $\{y \in \mathcal{E}:I(y) \leq M \}$ is  compact  in $\mathcal{E}$, for any $M > 0$.

\begin{defi} Let $\ds \{I^x\}_{x \in \mathcal{E}_0}$ be a family of good rate functions on $\mathcal{E}$ and let $\{Y_\epsilon^x\,;\ \epsilon > 0,\ x \in \mathcal{E}_0\}$ be a family of $\mathcal{E}$-valued random variables
Moreover, let $\gamma:(0,+\infty)\to (0,1)$, with $\gamma(\e)\to 0$, as $\e\to 0$. The family  of $\mathcal{E}$-valued random variables $\{Y_\epsilon^x\}_{\epsilon > 0}$ satisfies the large deviation principle (LDP) on $\mathcal{E}$ with speed $\gamma(\epsilon)$ and action functional $I^x$, uniformly for $x$ in the set $B\subseteq \mathcal{E}_0$, if the following two conditions hold.
\begin{enumerate}\item[(i)] For any open set $E \subset \mathcal{E}$,
\begin{equation}\label{eq:LDP_lower}
\liminf_{\epsilon \to 0} \gamma(\epsilon) \log \inf_{x \in B} \mathbb{P}(Y_\epsilon^x \in E) \geq - \sup_{x \in B} I^x(E):= -\sup_{x \in B} \inf_{y \in E}I^x(y).
\end{equation}
\item[(ii)] For any closed set $F \subset \mathcal{E}$,
\begin{equation}\label{eq:LDP_upper}
\limsup_{\epsilon \to 0} \gamma(\epsilon) \log \sup_{x \in B} \mathbb{P}(Y_\epsilon^x \in F) \leq - \inf_{x \in B} I^x(F):= -\inf_{x \in B} \inf_{y \in F}I^x(y).
\end{equation}
\end{enumerate}
\end{defi}

\begin{defi}
Let $\ds \{I^x\}_{x \in \mathcal{E}_0}$ be a family of good rate functions on $\mathcal{E}$ and let $B\subseteq \mathcal{E}_0$. The family of $\mathcal{E}$-valued random variables $\{Y_\epsilon^x\}_{\epsilon > 0}$   satisfies the  Laplace principle on $\mathcal{E}$ with speed $\gamma(\epsilon)$ and action functional $I^x$, uniformly for $x$ in the set $B$, if for any continuous and bounded $h:\mathcal{E}\to \R$
\begin{equation}
\label{eq*}
\lim_{\epsilon \to 0}\, \sup_{x \in B} \le|\gamma(\epsilon) \log \E \exp\le(-\frac{h(Y_\epsilon^x)}{\gamma(\epsilon)} \r) + \inf_{y \in \mathcal{E}} \le(I^x(y)+ h(y) \r) \r| = 0.
\end{equation}
\end{defi}
The equivalence between the non-uniform versions  of the large deviation principle and the Laplace principle is a well known fact. Recently, in  \cite{salins:171207231S} general results on the equivalence between the uniform versions of the large deviation principle and the Laplace principle have been investigated. 

\medskip

In order to study the problem of the exit of the solution of equation \eqref{eq:main} from a domain in $H_\mu$, we need a large deviation principle that is uniform with respect to $x$ on any bounded set of $\mathcal{E}_0 = H$.  In fact, since $H$ is a Hilbert space, and in particular reflexive, the weak convergence approach for SPDEs, as described in \cite{budhiraja2008},  allows us to prove a Laplace principle that is uniform on bounded sets.

The following proposition of \cite{salins:171207231S} then shows that the uniform Laplace principle implies the uniform large deviation principle. 
\begin{proposition}\label{prop:equiv}
Suppose that $\mathcal{E}_0$ is a reflexive Banach space and let $B \subset \mathcal{E}_0$ be a closed, bounded set. Moreover, assume the following conditions  hold.
\begin{enumerate}[(i)]
\item For any $s \geq 0$, the set $\Lambda_{s,B}:= \bigcup_{x \in B} \Phi^x(s)$ is compact in $\mathcal{E}$, where
\begin{equation*}
\Phi^x(s):= \{y \in \mathcal{E}:I^x(y) \leq s\}.
\end{equation*}
\item The mapping $x \mapsto \Phi^x(s)$ is weakly continuous in the Hausdorff metric,  for any $s \geq 0$. Namely, 
\begin{equation}\label{eq:hausdorff_cond}
x_n \rightharpoonup x,\  \text{ as } n \to \infty \implies \lim_{n \to \infty} \lambda(\Phi^{x_n}(s),\Phi^x(s)) = 0,
\end{equation}
where, for $A_1,A_2 \in \mathcal{E}$,
\begin{equation*}
\lambda(A_1,A_2) := \max \le\{\, \sup_{y \in A_1} \mathrm{dist}_{\mathcal{E}}(y,A_2), \sup_{y \in A_2} \mathrm{dist}_{\mathcal{E}}(y,A_1)\, \r\}.
\end{equation*}
\end{enumerate}
 Then, any family of $\mathcal{E}$-valued random variables $\{Y_\epsilon^x\}_{\epsilon > 0}$  that satisfies the Laplace principle on $\mathcal{E}$ with speed $\gamma(\epsilon)$ and action functional $I^x$, uniformly for $x \in B$, also satisfies the large deviation principle with the same speed and action functional, uniformly for $x \in B$. 
\end{proposition}

\section{Main results and description of the methods}\label{sec:results}
We are here   interested in the validity of a large deviation principle for the family $\{\mathcal{L}(u_\epsilon^x) \}_{\epsilon \in (0,1] }$, as $\epsilon \to 0$, where $u_\epsilon^x$ is the solution to the equation $\eqref{eq:main}$ with initial condition $x \in H$. 

In \cite{cerrai2011}, equation $\eqref{eq:main}$ was studied with $\alpha(\e)=\beta(\e)\equiv 1$, and it was shown that for every $\delta > 0$, the solutions converge in $L^p(\Omega;C([\delta,T];H_\mu))$ to the solution of the one-dimensional stochastic differential equation $\eqref{eq:averaged_SDE}$. Therefore, if 
\[\lim_{\e\to 0}\a(\e)=\lim_{\e\to 0}\beta(\e)=0,\]
 thanks to the bounds $\eqref{eq:Bconv_bound}$, $\eqref{eq:Qconv_bound}$  and $\eqref{eq:apriori_alpha0}$, the solution $u_\epsilon^x$ will converge in the space $L^p(\Omega;C([\delta,T];H_\mu))$ to the solution of the deterministic one-dimensional differential equation,
\begin{equation}\label{eq:ODE}
\ds \frac{du}{dt} = \bar{F}(t,u(t)), \ \ \ \ \ u(0) = \langle x, \mu \rangle.
\end{equation}

In what follows, we shall assume that the following conditions are satisfied
\begin{hyp}\label{hyp:nondeg}
\begin{enumerate}
\item[(i)] We have
\begin{equation}
\label{lim}
\lim_{\e\to 0}\a(\e)=\lim_{\e\to 0}\beta(\e)=0,\ \ \ \ \ \lim_{\e\to 0}\frac{\beta(\e)}{\a(\e)}=\bar{\rho} \in\,[0,+\infty].
\end{equation}
\item[(ii)] For every $t\geq 0$, $w \in\,\mathbb{R}$ and $\rho \in\,[0,+\infty]$, we define
\begin{equation}
\label{f5}
\mathcal{H}_{\rho}(t,w)=\frac 1{(1+\rho)^2}\le[\le|\sqrt{Q}\le[G(t,w)m \r]\r|_H^2 + \rho^2\,\le|\delta_0 \sqrt{B}\le[\Sigma(t) N_{\delta_0}^* m\r] \r|_Z^2\r],
\end{equation}
where $m$ is the density of the invariant measure $\mu$.
Then, if $\bar{\rho}$ is the constant introduced in \eqref{lim}, we have
\begin{equation}
\label{nondeg}
\inf_{(t,w) \in [0,\infty)\times \R} \mathcal{H}_{\bar{\rho}}(t,w)>0.
\end{equation}
\end{enumerate}
\end{hyp}

\medskip

Now, for every $0\leq \d<T$, we denote by $\Psi_{\d, T}$ the subset of  $C([\delta,T];H_\mu)$ containing all functions $u \in C([\delta,T];H_\mu)$ that are absolutely continuous in $t$ and are constant in the spatial variable $\xi$. Then, if $u \in \Psi_{\delta,T}$, we define 
\begin{equation}\label{eq:action}
\ds{I_{\delta,T}^x(u) = \inf_{\substack{w \in C([0,T];\R) \\ w(0) = \langle x, \mu \rangle,\ w|_{[\delta,T]} = u}} \frac{1}{2} \int_0^T  \frac{ \le|{w}^\prime(t) - \bar{F}(t,w(t)) \r|^2}{\mathcal{H}_{\bar{\rho}}(t,w(t))}\,dt.}
\end{equation}
For any other $u \in\, C([\delta,T];H_\mu)$, we set $I_{\delta,T}^x(u) = +\infty$.

We will show that, in fact, the laws of the family $\{u_\epsilon^x(t)\}_{\epsilon \in (0,1]}$ satisfy a large deviation principle in the space $C([\delta,T];H_\mu)$, with respect to the action functional $I_{\delta,T}^x$. 

\begin{theorem}\label{theor:main} Assume that all Hypotheses \ref{hyp:longterm} to \ref{hyp:nondeg} are satisfied.
Fix any $T >0$ and $0 < \delta < T$ and let $u_\epsilon^x$ denote the solution to equation $\eqref{eq:main}$, with initial condition $x \in H$. Moreover, let us define
\begin{equation}
\label{gamma}
\gamma(\e)=:\le(\a(\e)+\beta(\e)\r)^2,\ \ \ \ \e>0.\end{equation}
Then, the following facts hold.
\begin{enumerate}
\item[(i)] The family $\{\mathcal{L}(u_\epsilon^x)\}_{\epsilon \in (0,1]}$ satisfies a large deviation principle in $C([\delta,T];H_\mu)$ with speed $\gamma(\e)$ and action functional $I_{\delta,T}^x$, uniformly for $x$ in any closed, bounded subset of $H$.
 \item[(ii)] If in addition $x$ is constant, then $\{\mathcal{L}(u_\epsilon^x)\}_{\epsilon \in (0,1]}$ satisfy a large deviation principle in $C([0,T];H_\mu)$ with speed $\gamma(\e)$ and action functional $I_{0,T}^x$, uniformly for $x$ in any closed, bounded subset of $H$.
 \end{enumerate} 
\end{theorem}

To prove Theorem \ref{theor:main}, we follow the weak convergence approach, and instead prove the validity of a uniform Laplace principle, as described in Theorem \ref{theor:laplace}. This, combined with Proposition \ref{prop:equiv} yields a uniform large deviation principle, with same rate and same action functional. 

\medskip

First, we introduce some notation. We denote by $V$  the product Hilbert space $H \times Z$, endowed with the inner product 
\[\langle v_1,v_2 \rangle_V :=  \langle h_1,h_2 \rangle_H + \langle z_1,z_2 \rangle_Z,\]
for every $v_1=(h_1,z_1), \  v_2 = (h_2,z_2) \in V$.
 Next, we define the linear operator,
\begin{equation*}
Sv = (Qh,Bz),\ \ \ \ v=(h,z) \in\,V.
\end{equation*} 
 Notice that $S \in \mathcal{L}^+(V)$ and the process $w^S(t) := (w^Q(t),w^B(t))$, $t\geq 0$ is an $S$-Wiener process. Next, we let $\mathcal{P}(V)$ be the set of predictable processes in $L^2(\Omega \times [0,T];V)$. For every fixed $M>0$, we define
\[
S^M(V) := \le\{u \in L^2(0,T;V): \int_0^T |u(s)|_V^2\, ds \leq M\r\},\]
and
\[\mathcal{P}^M(V) := \le\{\varphi \in \mathcal{P}(V):\varphi \in S^M(V), \ \mathbb{P}-a.s. \r\} \]

For any $\varphi(t) = (\varphi_H(t),\varphi_Z(t)) \in \mathcal{P}^M(V)$, we denote by $u_\epsilon^{x,\varphi}$ the unique mild solution of the controlled stochastic PDE
\begin{equation}\label{eq:control_PDE}
\begin{cases} 
\ds{\frac{\partial u}{\partial t}(t,\xi) = \e^{-1}\, \mathcal{A}u(t,\xi) + f(t,\xi,u(t,\xi))} \\ 
\ds {\qquad \qquad \quad  +\frac{\a(\e)}{\sqrt{\gamma(\e)}}\, g(t,\xi,u(t,\xi))\le[\sqrt{Q}\varphi_H(t,\xi)+ \sqrt{\gamma(\epsilon)}\, \frac{\partial w^Q}{\partial t}(t,\xi) \r], } & \ds{\xi \in \mathcal{O}, }\\[10pt] \ds{ \frac{\partial u}{\partial \nu}(t,\xi) =   \e\,{\beta(\e)}{\sqrt{\gamma(\e)}}\,\sigma(t,\xi) \le[\sqrt{B}\varphi_Z(t,\xi)+ \sqrt{\gamma(\e)}\,\frac{\partial w^B}{\partial t}(t,\xi)  \r],}  & \ds{\xi \in \partial \mathcal{O},} \\[10pt] \ds{u(0,\xi) = x(\xi),} & \ds{\xi \in \mathcal{O}, }\end{cases}
\end{equation}
where $\gamma(\e)$ is the function defined in \eqref{gamma}.  Moreover, we denote by $u^{x,\varphi}$ the unique solution of the  random ODE
\begin{equation}\label{eq:control_ODE}
\ds \frac{du}{dt} = \bar{F}(t,u(t)) + \frac{1}{1+\bar{\rho}}\,\bar{G}(t,u(t)) \le[\sqrt{Q}\varphi_H(t)\r] + \frac{\bar{\rho}}{1+\bar{\rho}}\,\bar{\Sigma}(t)\le[\sqrt{B}\varphi_Z(t)\r],\ \ \ \ 
u(0) = \langle x, \mu \rangle.
\end{equation}
We will prove the well-posedness of both of these equations in the next section. In what follows, we denote 
\[\ds \mathcal{G}_\delta(x,\varphi):= u^{x,\varphi}|_{[\delta,T]}.\]

\begin{theorem}\label{theor:laplace}
For any $x \in H$, $ 0 < \delta < T$ and $u \in C([\delta,T];H_\mu)$, let
\begin{equation}\label{eq:action_abstract}
\hat{I}_{\delta,T}^x(u) := \inf_{\substack{\varphi \in L^2(0,T;V) \\ \mathcal{G}_\delta(x,\varphi) =u  }} \frac{1}{2}  \int_0^T |\varphi(s)|^2_{V} ds.
\end{equation}
 Suppose that the following conditions hold.
\begin{enumerate}[(i)]
\item If $B \subset H$ is a closed and bounded set, then for every $M < \infty$ the set
\begin{equation*}
F_{M,B}:= \le\{ u \in C([\delta,T];\R):u = \mathcal{G}_\delta(x,\varphi),\ \varphi \in S^M(V),\ x \in B   \r\},
\end{equation*}
is compact in $C([\delta,T];H_\mu)$.
\item If $\{\varphi_\epsilon \}_{\epsilon > 0} \subset \mathcal{P}^M(V)$ is any family that converges in distribution, as $\epsilon \to 0$, to some $\varphi \in \mathcal{P}^M(V)$ with respect to the weak topology of $L^2(0,T;V)$, and if $\{x_\epsilon\}_{\epsilon > 0}\subset H$ is any family that converges weakly in $H$, as $\epsilon \to 0$, to some $x \in H$, then the family $\{u_\epsilon^{x_\epsilon,\varphi_\epsilon}\}_{\epsilon > 0}$ converges in distribution, as $\epsilon \to 0$, to $u^{x,\varphi}$ in the space $C([\delta,T];H_\mu)$, endowed with the strong topology.
\item For every $u \in C([\delta,T;H_\mu])$, the mapping $x \mapsto \hat{I}_{\delta,T}^x(u)$ is weakly lower semicontinuous from $H$ into $[0,+\infty]$. 
\end{enumerate}

Then the family $\{\mathcal{L}(u_\epsilon^x)\}_{\epsilon>0}$ satisfies a Laplace principle in $C([\delta,T];H_\mu)$, with speed $\gamma(\e)$ and action functional $\hat{I}_{\delta,T}^x$, uniformly in $x$ on any closed, bounded subset of $H$. 
Moreover, for any closed bounded set $B \subset H$ and any $s\geq 0$, the set 
\begin{equation*}
\Lambda_{s,B} := \bigcup_{x \in B} \le\{u \in C([\delta,T];H_\mu):\hat{I}_{\delta,T}^x(u) \leq s\r\}
\end{equation*}
is compact in $C([\delta,T];H_\mu)$.
\end{theorem}

\begin{remark}\label{rem:action} \
In the  theorem above we allow the uniformity of the Laplace principle for initial conditions $x$  in closed and bounded sets $B \subset H$ (rather than compact sets). This is possible  by simply changing the topology of $H$ to the weak topology. Specifically, we require that the mapping $x \mapsto \hat{I}_{\delta,T}(u)$ is weakly lower semicontinuous and that condition (ii) must hold for any sequence $\{x_\epsilon\}_{\epsilon > 0}$ converging weakly to $x$. The reason why we can prove this stronger form of condition (ii) is because the limiting equation $\eqref{eq:ODE}$ is finite dimensional. If, for example, the averaging were only to occur in some but not all of the coordinates, then we would not have this property and condition (ii) would certainly fail for $x_\epsilon$ converging to $x$ only weakly. 

\begin{flushright}
$\Box$
\end{flushright}

\end{remark}

\section{Proof of Theorem \ref{theor:main}}\label{sec:proof}
\subsection{Well-Posedness of the skeleton equations}
\begin{proposition}\label{prop:wellposed} Assume that Hypotheses \ref{hyp:coeff} and \ref{hyp:eigenvalues} hold, and fix any $\epsilon,M > 0$ and  $p \geq 1$. Then for any $\varphi = (\varphi_H,\varphi_Z) \in \mathcal{P}^M(V)$ and $x \in H$, equation $\eqref{eq:control_PDE}$ has a unique adapted mild solution, $u_\epsilon^{x,\varphi} \in L^p(\Omega;C([0,T];H))$. Furthermore, if \eqref{lim} holds,  we have
\begin{equation}\label{eq:apriori}
\sup_{\epsilon \in (0,1]}\mathbb{E}\sup_{0 \leq t \leq T} |u_\epsilon^{x,\varphi}(t)|_H^p \leq c_{p,T, M}\le( 1 + |x|_H^p\r).
\end{equation}
\end{proposition}
\begin{proof}
The well-posedness of equation \eqref{eq:control_PDE} follows from a fixed point argument in the space of adapted processes in $L^p(\Omega;C([0,T];H))$. For $u \in L^p(\Omega;C([0,T];H))$, we define
\[\begin{array}{l}
\ds{
\mathcal{K}_\epsilon u(t)  := e^{t \frac A\e}x + \int_0^t e^{(t-s)\frac A\e} F(s,u(s))ds +  \a(\epsilon)\, w_{A,Q}^\epsilon(u)(t) +\beta(\epsilon)\, w_{A,B}^\epsilon(t) }\\
\vspace{.1mm}\\
\ds{ + \frac{\a(\e)}{\sqrt{\gamma(\e)}}\,\int_0^t e^{(t-s)\frac A\e} G(s,u(s))\le[\sqrt{Q}\varphi_H(s)\r]ds }\\
\vspace{.1mm}\\
\ds{ + \frac{\beta(\e)}{\sqrt{\gamma(\e)}}\,(\delta_0 - A) \int_0^t e^{(t-s)\frac A\e} N_{\delta_0} \le[\Sigma(s) \sqrt{B}\varphi_Z(s) \r]ds.}
\end{array}\]
We show that $\mathcal{K}_\epsilon$ is Lipschitz continuous from  $L^p(\Omega;C([0,T];H))$ into itself, with Lipschitz constant going to $0$ as $T \to 0$. This clearly implies the well-posedness of equation \eqref{eq:control_PDE} in $L^p(\Omega;C([0,T];H))$.

Thanks to $\eqref{eq:Bconv_bound},\eqref{eq:Qconv_bound}$ and $\eqref{eq:Qconv_lip}$, since $F(t,\cdot):H \to H$ is Lipschitz-continuous, it suffices to show that the mapping $\Gamma_\epsilon$, defined by
\[\begin{array}{l}
\ds{\Gamma_\epsilon(u)(t)  = \frac{\a(\e)}{\sqrt{\gamma(\e)}}\,\int_0^t e^{(t-s)\frac A\e} G(s,u(s))\sqrt{Q}\varphi_H(s)ds}\\
\vspace{.1mm}\\
\ds{ + \frac{\beta(\e)}{\sqrt{\gamma(\e)}}\,(\delta_0 - A)\int_0^t e^{(t-s)\frac A\e} N_{\delta_0} \le[\Sigma(s)\sqrt{B}\varphi_Z(s) \r]ds,}
\end{array}\]
maps $L^p(\Omega;C([0,T];H))$ into itself and is Lipschitz continuous with Lipschitz constant going to $0$ as $T \to 0$. To this purpose, we define $ \zeta = \frac{2\rho}{\rho-2}$, where $ \rho < \frac{2d}{d-2}$ satisfies $\eqref{eq:rhohyp}$. Since $\zeta < d$ thanks to  $\eqref{eq:Glemma_3}$, for any $u,v \in L^p(\Omega;C([0,T];H))$, we have
\begin{equation}
\label{f100}
\begin{array}{l}
\ds{ \sup_{0 \leq t \leq T}  \le|\int_0^t  e^{(t-s)\frac A\e}\le(G(s,u(s))-G(s,v(s))\r)\le[\sqrt{Q}\varphi_H(s) \r]ds \r|_H^p }\\
\vspace{.1mm}\\
\ds{ \leq c\,  \sup_{0 \leq t \leq T} \le( \int_0^t \le(\le((t-s)/\e\r)^{-\frac d{2\zeta}}+1\r) |u(s)-v(s)|_H|\varphi_H(s)|_H ds  \r)^p}\\
\vspace{.1mm}\\
\ds{\leq c_p\, (T^{\frac p2}+T^{(1-\frac d\zeta)\frac p2})\le[ \le( \int_0^T |\varphi_H(s)|_H^2 ds \r)^{\frac p2} \sup_{0 \leq t \leq T} |u(s)-v(s)|_H^p \r] }\\
\vspace{.1mm}\\
\ds{ \leq c_{p,M,T}\,  \sup_{0 \leq t \leq T} |u(s)-v(s)|_H^p,}
\end{array}\end{equation}
where, in the last step, we used the fact that $|\varphi_H|_{L^2(0,T;H)}^2\leq |\varphi|_{L^2(0,T;V)}^2 \leq M$, $\mathbb{P}$-a.s.. 

To conclude our proof of the well-posedness, we show that the second term in $\Gamma_\epsilon$ is in $L^p(\Omega;C([0,T];H))$. Due to $\eqref{eq:Neuman_bounded}$ and $\eqref{eq:sem_domain}$, the operator 
\begin{equation}
\label{sro}
S_\rho := (\delta_0 - A)^{\frac{3-\rho}4} N_{\delta_0}
\end{equation}
 belongs to $\mathcal{L}(Z,H)$, for any $\rho > 0$. Therefore, for any $t>0$, we have
\begin{equation}\label{eq:Srho_identity}
(\delta_0 - A) e^{tA}N_{\delta_0}    = e^{\frac t2 A}(\delta_0-A)^{\frac{1+\rho}4} e^{\frac t2 A} S_\rho,
\end{equation}
and
\[\begin{array}{l}
\ds{\le|(\delta_0 - A)  \int_0^t e^{(t-s)\frac A\e} N_{\delta_0}[ \Sigma(s)\sqrt{B}\varphi_Z(s)] ds \r|_H}\\
\vspace{.1mm}\\
\ds{  
\leq \int_0^t \le|e^{(t-s) \frac A{2\e}} (\delta_0 - A)^{\frac{1+\rho}4} e^{(t-s) \frac A{2\e}} S_\rho \le[\Sigma(s)\sqrt{B}\varphi_Z(s)\r] \r|_H ds}\\
\vspace{.1mm}\\
\ds{\leq c \int_0^t  \le[1 +(t-s)^{-\frac{1+\rho}{4}} \r] |\varphi_Z(s)|_Z  ds.}
\end{array}\]
Thus, by taking the $p$-th moment and choosing $\rho < 1$, we get
\begin{equation}
\label{f200}
\begin{array}{l}
\ds{ \sup_{0 \leq t \leq T}  \le|(\delta_0-A) \int_0^t e^{(t-s)\frac A\e} N_{\delta_0}[\Sigma(s)\sqrt{B} \varphi_Z(s)]ds \r|^p_H}\\
\vspace{.1mm}\\
\ds{  \leq c\, \sup_{0 \leq t \leq T} \le(\int_0^t [1+s^{-\frac{1+\rho}2}]ds \r)^{\frac p2} \le(\int_0^t |\varphi_Z(s)|_Z^2ds \r)^{\frac p2}  \leq c_p\,M^{\frac p2} (T^{\frac p2}+T^{\frac{p(1-\rho)}4}).}
\end{array}\end{equation}
\medskip

Next, we prove that estimate \eqref{eq:apriori} holds.
To this purpose, we first remark that due to \eqref{lim}
\begin{equation}
\label{lim*}
\lim_{\e\to 0} \frac{\a(\e)}{\sqrt{\gamma(\e)}}=\frac 1{1+\bar{\rho}} \in\,[0,1],\ \ \ \lim_{\e\to 0} \frac{\beta(\e)}{\sqrt{\gamma(\e)}}=\frac {\bar{\rho}}{1+\bar{\rho}} \in\,[0,1].\end{equation}
In particular, both $\a(\e)/\sqrt{\gamma(\e)}$ and $\beta/\sqrt{\gamma(\e)}$ remain uniformly bounded, with respect to $\e \in\,(0,1]$.

Thus, thanks to \eqref{lim*}, by proceeding as in the proofs of \eqref{f100} and \eqref{f200}, due to \eqref{eq:Bconv_bound} and \eqref{eq:Qconv_bound}, we have
\[\mathbb{E}\,\sup_{0\leq t\leq T}|u^{x,\varphi}_\e(t)|_H^p\leq c_p\,\le(1+|x|_H^p\r)+c_{p,M,T}\le(1+\mathbb{E}\,\sup_{0\leq t\leq T}|u^{x,\varphi}_\e(t)|_H^p\r),\ \ \ \ \e \in\,(0,1],\]
for some constant $c_{p,M,T}>0$ such that $c_{p,M,T}\to 0$, as $T\to 0$. This means that there exists $T_0>0$ such that 
\[\sup_{\e \in\,(0,1]}\,\mathbb{E}\,\sup_{0\leq t\leq T_0}|u^{x,\varphi}_\e(t)|_H^p\leq 2\,c_p\,\le(1+|x|_H^p\r).\]
By a bootstrap argument, this yields $\eqref{eq:apriori}$.

\end{proof}
\begin{proposition}
Assume that Hypothesis \ref{hyp:coeff} hold, and fix any $M > 0$. Then, for any $\varphi = (\varphi_H,\varphi_Z) \in \mathcal{P}^M(V)$, the random differential equation $\eqref{eq:control_ODE}$ has a unique adapted solution, $u^{x,\varphi} \in L^p(\Omega;C([0,T];\R))$, for any $T>0$ and $p\geq 1$.
\end{proposition}
\begin{proof}
As before, existence and uniqueness follows from the Lipschitz continuity of the mapping $\mathcal{K}:L^p(\Omega;C([0,T];\R)) \to L^p(\Omega;C([0,T];\R))$, defined by
\[\begin{array}{l}
\ds{\mathcal{K} u(t) = \langle x, \mu \rangle + \int_0^t \bar{F}(s,u(s))ds}\\
\vspace{.1mm}\\
\ds{+\frac 1{1+\bar{\rho}}\int_0^t\bar{G}(s,u(s))\le[\sqrt{Q}\varphi_H(s)\r]ds + \frac{\bar{\rho}}{1+\bar{\rho}} \int_0^t \bar{\Sigma}(t)\le[\sqrt{B}\varphi_Z(s)\r] ds.}
\end{array}\]
Let $u,v \in L^p(\Omega;C([0,T];\R))$. Due to $\eqref{eq:Gbar_lip}$ and the fact that $\varphi \in \mathcal{P}^M(V)$, we have
\[\begin{array}{l}
\ds{\mathbb{E} \sup_{0 \leq t \leq T} \le|  \int_0^t \le( \bar{G}(s,u(s))  -\bar{G}(s,v(s)) \r)\le[ \sqrt{Q}\varphi_H(s)\r]  ds \r|^p }\\
\vspace{.1mm}\\
\ds{ \leq c\,\mathbb{E} \sup_{0 \leq t \leq T} \le( \int_0^t |\varphi_H(s)|_H |u(s)-v(s)|\,ds \r)^p \leq c_p\, T^{\frac p2} M^{\frac p2} \mathbb{E}\sup_{0 \leq t \leq T} |u(t)-v(t)|^p.}
\end{array}\]
Moreover, due to $\eqref{eq:Neuman_bounded}$, we easily have
\[\begin{array}{l}
\ds{\mathbb{E} \sup_{0 \leq t \leq T} \le|\int_0^t \bar{\Sigma}(t)\le[\sqrt{B}\varphi_Z(s)\r]ds  \r|^p\!\! = \mathbb{E} \sup_{0 \leq t \leq T} \le| \delta_0 \int_0^t\langle N_{\delta_0}\le(\Sigma(s)[\sqrt{B}\varphi_Z(s)]\r), \mu \rangle ds \r|^p \!\! \leq c\, T^{\frac p2}M^{\frac p2}.}\\
\vspace{.1mm}\\
\ds{}
\end{array}\]
Since $\bar{F}(t,\cdot):\R \to \R$ is Lipschitz continuous, we conclude that $\mathcal{K}$ is Lipschitz continuous from $L^p(\Omega;C([0,T];\R)$ into itself, and the well-posedness of equation $\eqref{eq:control_ODE}$ follows.
\end{proof}

\subsection{Convergence}
Clearly, in Theorem \ref{theor:laplace} Condition (i) follows immediately from Condition (ii). On the other hand, due to  Skorokhod's theorem, Condition (ii) in Theorem \ref{theor:laplace} follows from the following convergence result.

\begin{proposition}\label{prop:conv}
Assume that Hypotheses \ref{hyp:longterm}, \ref{hyp:coeff} and \ref{hyp:eigenvalues} hold. Moreover, assume that \eqref{lim} holds. Suppose that $\{x_\epsilon\}_{\epsilon > 0} \subset H$ converges weakly to $x \in H$, as $\epsilon \to 0$, and suppose that $\{\varphi_\epsilon\}_{\epsilon > 0} \subset \mathcal{P}^M(V)$ converges weakly in $L^2(0,T;V)$ to $\varphi$, as $\epsilon \to 0$, $\mathbb{P}$-a.s. Then for any $\delta > 0$ and $p \geq 1$,
\begin{equation}\label{eq:conv_limit}
\lim_{\epsilon \to 0} \mathbb{E} \sup_{\delta \leq t \leq T} |u_\epsilon^{x_\epsilon,\varphi_\epsilon} - u^{x,\varphi}|_{H_\mu}^p = 0.
\end{equation}
\end{proposition}
\begin{proof}
We denote $\varphi_\epsilon = (\varphi_H^\epsilon,\varphi_Z^\epsilon)$ and $\varphi = (\varphi_H,\varphi_Z)$. We first write
\[\begin{array}{l}
\ds{u_\epsilon^{x_\epsilon,\varphi_\epsilon}(t)  - u^{x,\varphi}(t) =   \le( e^{t \frac A\e}x_\epsilon - \langle x, \mu \rangle \r) +  \a(\epsilon)\, w_{A,Q}^\epsilon(u_\epsilon^{x_\epsilon,\varphi_\epsilon})(t) +\beta(\epsilon)\, w_{A,B}^\epsilon(t) }\\
\vspace{.1mm}\\
\ds{ + \le( \int_0^t e^{(t-s)\frac A\e} F(s,u_\epsilon^{x_\epsilon,\varphi_\epsilon}(s))ds  -\int_0^t \bar{F}(s,u^{x,\varphi}(s)) ds\r) }\\
\vspace{.1mm}\\
\ds{ +\frac{\a(\e)}{\sqrt{\gamma(\e)}} \int_0^t e^{(t-s)\frac A\e} G(s,u_\epsilon^{x_\epsilon,\varphi_\epsilon}(s))\le[\sqrt{Q}\varphi_H^\epsilon(s)\r]ds - \frac 1{1+\bar{\rho}}\int_0^t \bar{G}(s,u^{x,\varphi}(s)) \le[\sqrt{Q}\varphi_H(s)\r] ds   }\\
\vspace{.1mm}\\
\ds{ +\frac{\beta(\e)}{\sqrt{\gamma(\e)}}  (\delta_0 - A) \int_0^t e^{(t-s)\frac A\e} N_{\delta_0} \le[\Sigma(s) \sqrt{B}\varphi_Z^\epsilon(s) \r]ds - \frac {\bar{\rho}}{1+\bar{\rho}}\int_0^t \bar{\Sigma}(s)\le[\sqrt{B}\varphi_Z(s)\r]\,ds }\\
\vspace{.1mm}\\
\ds{ =: \le(e^{t \frac A\e}x_\epsilon - \langle x, \mu \rangle \r) +  \a(\epsilon)\, w_{A,Q}^\epsilon(u_\epsilon^{x_\epsilon,\varphi_\epsilon})(t) +\beta(\epsilon)\, w_{A,B}^\epsilon(t) + \sum_{i=1}^3 I^i_\epsilon(t).}
\end{array}\]
Thanks to estimates $\eqref{eq:Bconv_bound}$, $\eqref{eq:Qconv_bound}$ and $\eqref{eq:apriori}$, as well as Lemmas \ref{lem:Fterm}, \ref{lem:Gterm}, and \ref{lem:Sterm} (where a-priori bounds for the terms $I^i_\e(t)$ are proven), there exists some non-negative function $r_{T, p}(\epsilon)$ going to $0$, as $\epsilon \to 0$, such that
\begin{equation}\label{eq:conv_Gron}
\begin{array}{l}
\ds{
 \mathbb{E} \sup_{\delta \leq t \leq T}  |u_\epsilon^{x_\epsilon,\varphi_\epsilon}(t)  - u^{x,\varphi}(t)|_{H_\mu}^2}\\
\vspace{.1mm}\\
\ds{ \leq \sup_{\delta \leq t \leq T} \le|e^{t \frac A\e}x_\epsilon - \langle x, \mu \rangle \r|_{H_\mu}^2 + r_{T,p}(\epsilon)  + c_{T} \int_\delta^T \mathbb{E} \sup_{\delta \leq s \leq t} \le|u_\epsilon^{x_\epsilon,\varphi_\epsilon}(s)-u^{x,\varphi}(s)\r|_{H_\mu}^2.}
\end{array}\end{equation}
We have
\[\begin{array}{l}
\ds{\sup_{\delta \leq t \leq T} \big|e^{t \frac A\e}x_\epsilon  - \langle x, \mu \rangle \big|_{H_\mu}^2  \leq 2 \sup_{\delta \leq t \leq T} \big|e^{t \frac A\e}x_\epsilon - \langle x_\epsilon, \mu \rangle \big|_{H_\mu}^2 + 2 \big|\langle x_\epsilon-x,\mu \rangle \big|^2 }\\
\vspace{.1mm}\\
\ds{ \leq 2\, c\, e^{-\frac{2\gamma \delta}\e} |x_\epsilon|_{H_\mu} + 2 \big|\langle x_\epsilon-x,\mu \rangle \big|^2.}
\end{array}\]
Therefore, since the sequence $\{x_\epsilon\}_{\e>0}$ converges weakly to $x$ in $H_\mu$, we have
\[\lim_{\e\to 0}\,\sup_{\delta \leq t \leq T} \big|e^{t \frac A\e}x_\epsilon  - \langle x, \mu \rangle \big|_{H_\mu}^2=0.\]
This fact, together with  $\eqref{eq:conv_Gron}$ and Gronwall's Lemma, allows us to  conclude that \eqref{eq:conv_limit} holds for $p\geq 2$.
To obtain the result for $p > 2$, we use  estimate $\eqref{eq:apriori}$ and the dominated convergence theorem.
\end{proof}

In the next section  we show that for every $u \in\,C([\d,T];H_\mu)$ the mapping $x \in\,H\mapsto \hat{I}_{\delta,T}^x(u) \in\,[0,+\infty]$ is weakly lower semicontinuous. Due to the convergence result proved in Proposition \ref{prop:conv} and to Theorem \ref{theor:laplace}, this implies  that the family $\{u^x_\e\}_{\e>0}$ satisfies a uniform Laplace principle in $C([\d,T];H_\mu)$, with speed $\gamma(\e)$ and action functional $\hat{I}_{\delta,T}^x$.

\subsection{Conclusion}\label{sec:action}
In this section, we first show that under Hypothesis \ref{hyp:nondeg}, we have $\hat{I}_{\delta,T}^x = I_{\delta,T}^x$ where $I_{\delta,T}^x$ is the action functional defined in $\eqref{rem:action}$. Then, we show that the action functionals $I_{\delta,T}^x$ satisfy the properties required to extend the uniform Laplace principle into a uniform large deviation principle (see Proposition \ref{prop:equiv}). In particular, the mapping $x \mapsto \hat{I}^x_{\d, T}(u) \in\,[0,+\infty]$ is weakly lower semicontinuous, for every $u \in\,C([0,T];H_\mu)$. This will conclude the proof of Theorem \ref{theor:main}.

\begin{lemma} For every $0\leq a<b$, let us define 
\begin{equation*}
I_{a,b}(w) :=\frac{1}{2} \int_a^b  \frac{ \le|{w}^\prime(t) - \bar{F}(t,w(t)) \r|^2}{\mathcal{H}_{\bar{\rho}}(t,w(t))}\,dt.
\end{equation*}
Then, we have
\begin{equation}
\label{caratterizzazione}
\hat{I}_{\delta,T}^x(u) = I_{\delta,T}^x(u)= \inf_{\substack{w \in C([0,T];\R) \\ w(0) = \langle x, \mu \rangle,\ w|_{[\delta,T]} = u}} I_{0,T}(w).
\end{equation}
\end{lemma}

\begin{proof}
First, we observe that $\hat{I}_{\delta,T}^x(u) = \infty$, if $u(t,\xi)$ is any function depending on the spatial variable $\xi$. Next, we  notice  that  $\hat{I}_{\delta,T}^x(u)$ can be rewritten as 
\begin{equation*}
\hat{I}^x_{\delta,T}(u) = \inf_{\substack{w \in C([0,T];\R)  \\ w|_{[\delta,T]} = u}} \  \inf_{\substack{\varphi \in L^2(0,T;V) \\  u^{x,\varphi}=w}} \frac{1}{2}  \int_0^T |\varphi(s)|^2_{V} ds,
\end{equation*}
because the condition $\mathcal{G}_\delta(x,\varphi) = u$ does not constrain the values of $\varphi$ on the interval $(0,\delta)$.
 We suppose now that $w = u^{x,\varphi}$, for some $x \in H$ and $\varphi = (\varphi_H,\varphi_Z) \in L^2(0,T;V)$. Then, recalling that $\mu$ has a density $m \in L^\infty(\mathcal{O})$, we have
\[\begin{array}{l}
\ds{|{w}^\prime(t) -  \bar{F}(t,w(t)) | = \frac 1{1+\bar{\rho}}\,\le| \bar{G}(t,w(t))\le[\sqrt{Q}\varphi_H(t)\r] + \bar{\rho}\,\bar{\Sigma}(t)\le[\sqrt{B}\varphi_Z(t) \r] \r|}\\
\vspace{.1mm}\\
\ds{=\frac 1{1+\bar{\rho}}\, \le|\langle \varphi_H(t),\sqrt{Q}\le[G(t,w(t))m \r] \rangle_H +\bar{\rho}\,\delta_0 \langle \varphi_Z(t),  \sqrt{B} \le[\Sigma(t) N_{\delta_0}^*m \r] \rangle_Z \r|}\\
\vspace{.1mm}\\
\ds{ \leq \frac 1{1+\bar{\rho}}\,| \varphi_H(t) |_H\le( \big| \sqrt{Q}\le[G(t,w(t))m\r] \big|_H + \bar{\rho}\,\delta_0\, | \varphi_Z(t) |_Z  \big|  \sqrt{B} \le[\Sigma(t) N_{\delta_0}^*m \r] \big|_Z \r)}\\
\vspace{.1mm}\\
\ds{\leq \frac 1{1+\bar{\rho}}\,|\varphi(t)|_V \le( \big| \sqrt{Q}\le[G(t,w(t))m\r] \big|_H^2 + \bar{\rho}^2\,\delta_0^2\,\big|  \sqrt{B} \le[\Sigma(t) N_{\delta_0}^*m \r] \big|_Z^2 \r)^{1/2}}\\
\vspace{.1mm}\\
\ds{=|\varphi(t)|_V\,\sqrt{\mathcal{H}_{\bar{\rho}}(t,w(t))}.}
\end{array}\]
On the other hand, equality is achieved with the choice,
\[
\hat{\varphi}(t): =  \frac{1}{1+\bar{\rho}}\,\frac{ {w}^\prime(t)-\bar{F}(w(t))}{\mathcal{H}_{\bar{\rho}}(t,w(t))}\le(\sqrt{Q}\le[G(t,w(t))m \r],\bar{\rho}\,\delta_0\, \sqrt{B}\le[\Sigma(t) N_{\delta_0}^* m\r]\r).
\]
Notice that $\hat{\varphi}$ is well defined due to the non-degeneracy condition in  Hypothesis \ref{hyp:nondeg}. Moreover, it is easy to see that $w$ solves equation $\eqref{eq:control_ODE}$ with the control $\hat{\varphi}$, so that $u^{x,\hat{\varphi}} = w$. This minimizing choice of $\varphi = \hat{\varphi}$ gives rise to the action functional $I^x_{\d, T}$.
\end{proof}

Alternatively, we can write the action functional as
\begin{equation}\label{eq:action_alt}
 I_{\delta,T}^x(u)  = \inf_{\substack{w \in C([0,\delta];\R) \\ w(0) = \langle x, \mu \rangle,\ w(\delta)=u(\delta)}} I_{0,\delta}(w)  + I_{\delta,T}(u)  =: J_\d(x,u)+ I_{\delta,T}(u) 
\end{equation}
$J_\d(x,u)$ depends only on the initial condition $x \in H$ and the value of the path $u$ at $t = \delta$, while $I_{\delta,T}(u)$ only depends on the path $u$.

\begin{lemma}
\label{n10}
Assume $H$ is endowed with the weak topology. Then, for every $0<\d<T$, the mapping
$J_\d:H\times C([0,T];\mathbb{R})\to [0,+\infty)$ is continuous.
\end{lemma}
\begin{proof}
For every $x \in\,H$, $u \in\,C([0,T];\mathbb{R})$ and $\eta>0$, we denote by $w_\eta(x,u)$ a path in $C([0,T];\mathbb{R})$ such that 
\[w_\eta(x,u)(0)=\langle x,\mu\rangle,\ \ \ \ w_\eta(x,u)(\d)=u(\d),\ \ \ \ J_\d(x,u)\geq I_{0,\d}(w_\eta(x,u))-\frac \eta 4.\]
Moreover, for every $y \in\,H$, $v, w \in\,C([0,T];\mathbb{R})$ and $\d^\prime \in\,(0,\d)$, we denote by $\rho_{\d^\prime}(y,v,w)$ the path in $C([0,T];\mathbb{R})$ defined by
\[\ds{\rho_{\d^\prime}(y,v,w)(t)=}\begin{cases}
\ds{(\d^\prime-t)/\d^\prime\,\langle y,\mu\rangle+t/\d^\prime w(\d^\prime), } &  \ds{t \in\,[0,\d^\prime],}\\
\ds{w(t), } &  \ds{t \in\,[\d^\prime, \d-\d^\prime]}\\
\ds{(\d-t)/\d^\prime\,w(\d-\d^\prime)+(t-(\d-\d^\prime))/\d^\prime \,v(\d),}  &  \ds{t \in\,[\d-\d^\prime,\d].}
\end{cases}\]
Since $\rho_{\d^\prime}(y,v,w)$ and $w$ coincide in the interval $[\d^\prime,\d-\d^\prime]$, we have
\begin{equation}
\label{n1}
\begin{array}{l}
\ds{\le|I_{0,\d}(\rho_{\d^\prime}(y,v,w))-I_{0,\d}(w)\r|}\\
\vspace{.1mm}\\
\ds{\leq \le|I_{0,\d^\prime}(\rho_{\d^\prime}(y,v,w))\r|+\le|I_{\d-\d^\prime,\d^\prime}(\rho_{\d^\prime}(y,v,w))\r|+\le|I_{0,\d^\prime}(w)\r|+\le|I_{\d-\d^\prime,\d}(w)\r|.}
\end{array}
\end{equation}

Now, let us fix $x \in\,H$ , $u \in\,C([0,T];\mathbb{R})$ and $\eta>0$. Let $\{x_n\}_{n \geq 1}\subset H$ be a sequence weakly convergent to 
$x$ and let $\{u_n\}\subset C([0.T];\mathbb{R})$ be a sequence convergent to $u$. 
For every $n \in\,\mathbb{N}$ and $\d^\prime \in\,(0,\d)$, we have
\begin{equation}
\label{n2}
\begin{array}{l}
\ds{J_\d(x_n,u_n)\leq I_{0,\d}(\rho_{\d^\prime}(x_n,u_n,w_\eta(x,u)))}\\
\vspace{.1mm}\\
\ds{\leq I_{0,\d}(\rho_{\d^\prime}(x_n,u_n,w_\eta(x,u)))-I_{0,\d}(w_\eta(x,u))+J_\d(x,u)+\eta/4.}
\end{array}
\end{equation}
Since the sequences $\{x_n\}_{n\geq 1}$ and $\{u_n\}_{n\geq 1}$ are bounded and Hypothesis \ref{hyp:nondeg} holds true, we have
\[\begin{array}{l}
\ds{\le|I_{0,\d^\prime}(\rho_{\d^\prime}(x_n,u_n,w_\eta(x,u)))\r|
\leq  c\int_0^{\d^\prime}
\le(\frac{\le|w_\eta(x,u)(\d^\prime)-\langle x_n,\mu\rangle\r|_H}{\d^\prime}+1\r)^2\,dt}\\
\vspace{.1mm}\\
\ds{\leq  c\,\le(\frac{|w_\eta(x,u)(\d^\prime)-w_\eta(x,u)(0)|^2}{\d^\prime}+\frac{|\langle x_n-x,\mu\rangle|^2}{\d^\prime}+\d^\prime\r).}
\end{array}\]
Analogously,
\[\begin{array}{l}
\ds{\le|I_{\d-\d^\prime,\d}(\rho_{\d^\prime}(x_n,u_n,w_\eta(x,u)))\r|
\leq  c\int_{\d-\d^\prime}^\d
\le(\frac{\le|w_\eta(x,u)(\d^\prime)-u_n(\d)\r|_H}{\d^\prime}+1\r)^2\,dt}\\
\vspace{.1mm}\\
\ds{\leq  c\,\le(\frac{|w_\eta(x,u)(\d-\d^\prime)-w_\eta(x,u)(\d)|^2}{\d^\prime}+\frac{|u_n(\d)-u(\d)|^2}{\d^\prime}+\d^\prime\r).}
\end{array}\]
Therefore, as $w_\eta(x,u) \in\,W^{1,2}(0,\d)$, we can find $\d^\prime_1>0$ such that
\[\le|I_{0,\d^\prime}(\rho_{\d^\prime}(x_n,u_n,w_\eta(x,u)))\r|+\le|I_{\d-\d^\prime,\d}(\rho_{\d^\prime}(x_n,u_n,w_\eta(x,u)))\r|\leq \eta/4,\ \ \ \ \d^\prime\leq \d^\prime_1.\]
Moreover, as $I_{0,\d}(w_\eta(x,u))<\infty$, we can find $\d^\prime_2>0$ such that
\[\le|I_{0,\d^\prime}(w_\eta(x,u))\r|+\le|I_{\d-\d^\prime,\d}(w_\eta(x,u))\r|\leq \eta/4,\ \ \ \ \d^\prime\leq \d^\prime_2.\]
Thus, if we pick $\bar{\d}^\prime=\min (\d^\prime_1,\d^\prime_2)$, thanks to \eqref{n1} we conclude that
\[\limsup_{n\to\infty}\le|I_{0,\d}(\rho_{\d^\prime}(x_n,u_n,w_\eta(x,u)))-I_{0,\d}(w_\eta(x,u))\r|<\frac 34 \eta.\]
Thanks to \eqref{n2}, this implies that there exists $n^1_\eta \in\,\mathbb{N}$ such that
\begin{equation}
\label{n3}
J_\d(x_n,u_n)\leq J_\d(x,u)+\eta,\ \ \ \ n\geq n^1_\eta.
\end{equation}
Next, we want to prove that there exists $n^2_\eta \in\,\mathbb{N}$ such that
\begin{equation}
\label{n4}
J_\d(x_n,u_n)\geq J_\d(x,u)-\eta,\ \ \ \ n\geq n^2_\eta.
\end{equation}
The proof of the inequality above follows the same line of the proof of inequality \eqref{n3}. Actually, as in \eqref{n2} we have
\[\begin{array}{l}
\ds{J_\d(x,u)\leq \le|I_{0,\d}(\rho_{\d^\prime}(x,u,w_\eta(x_n,u_n)))-I_{0,\d}(w_\eta(x_n,u_n))\r|-J_{0,\d}(x_n,u_n)+\eta/4.}
\end{array}\]
Then, by using the same arguments used above, we can find a sequence $\{\d^\prime_n\}_{n\geq 1}\subset (0,\d)$ such that
 \[\limsup_{n\to\infty}\le|I_{0,\d}(\rho_{\d_n^\prime}(x,u,w_\eta(x_n,u_n)))-I_{0,\d}(w_\eta(x_n,u_n))\r|<\frac 34 \eta,\]
 and \eqref{n4} follows.

\end{proof}

The continuity above is strictly due to the fact that $\delta > 0$. If $\delta = 0$ then certainly the mapping $x \mapsto I^x(u)$ is not continuous, since $I^x(u)$ is finite only if $u(0) = x$. However, the lemma above easily implies the following weaker condition, which is also true in the $\delta = 0$ case.

\begin{lemma}
For every sequence  $\{x_n\}_{n\geq 1}\subset H$, weakly convergent to some $x$, and for every  $u \in \{\varphi \in C([\delta,T];\R): I^x_{\delta,T}(\varphi) \leq s\}$, there exists a sequence $\{u_n\}_{n\geq 1}$ such that $u_n \to u$ in $C([\d,T];H)$ and
\begin{equation}\label{eq:equiv_cond2}
\limsup_{n \to  \infty} I_{\d,T}^{x_n}(u_n) \leq s.
\end{equation}
\end{lemma}

When $\delta > 0$, this is trivially satisfied by the sequence $u_n = u$ by the previous lemma. This condition can be used along with the following lemma to prove that the conditions of Proposition \ref{prop:equiv} is satisfied in the model we are studying.

\begin{lemma}
For any $\delta > 0, T > 0, s \geq 0$ and $x \in H$, we define the set
\[\Phi^x_{\delta,T}(s):= \{\varphi \in C([\delta,T];\R): I^x_{\delta,T}(\varphi) \leq s\}.
\]
Then, for any bounded set $B \subset H$, we have
\begin{equation*}
\lim_{r \to 0^+} \sup_{x \in B} \lambda(\Phi_{\delta,T}^x(s),\Phi_{\delta,T}^x(s+r)) = 0.
\end{equation*}

\end{lemma}
\begin{proof}
Fix an $\epsilon > 0$ and $s > 0$. We will show that there exist $r > 0$ small enough that for any $x \in B$ and $u \in \Phi^x_{\delta,T}(s+r)$, there exists $z_u \in \Phi^x_{\delta,T}(s)$ such that $|u-z_u|_{C([\delta,T];\R)} < \epsilon$.

Fix an $r > 0$. First we consider the case of $x \in B$ and $u \in \Phi^x_{\delta,T}(s+r)$ such that $I_{\delta,T}(u) > r$. For such a path $u$, we may consider the continuous path $z_u \in C([\delta,T];\R)$ defined by
\[
z_u(t)= \begin{cases}
u(t), & \text{ if } t \in [\delta,T^*], \\
u^{u(T^*)}(t-T^*), & \text{ if } t \in [T^*,T],
\end{cases}
\]
where 
\[T^*=T^*(u,r):=\inf \{t \in [\delta,T]:I_{t,T}(u) \leq r\}.
\]
Hence, $z_u \in \Phi_{\delta,T}^x(s)$.
 Moreover, since $W^{1,2}([\delta,T]) \hookrightarrow C([\delta,T];\R)$, it is easy to see that
\[\sup_{x \in B} \sup_{u \in \Phi^x_{\d,T}(s+r)}|u|_{C([\delta,T];\R)} < \infty.
\]
Thanks to the Lipschitz condition on $g$, this implies that
\[\sup_{x \in B} \sup_{u \in \Phi^x_{\d,T}(s+r)}|\mathcal{H}_{\bar{\rho}}(u)|_{C([\delta,T];\R)} < \infty.
\]
Next, for any $t \in [T^*,T]$ we have that
\[\begin{array}{l}
|u(t)-z_u(t)|^2 \leq  \ds \Big( \int_{T_*}^t |u^\prime(s)-\bar{F}(s,u(s))|ds + \int_{T_*}^t |\bar{F}(s,u(s))-\bar{F}(s,z_u(s))|ds \Big)^2\vspace{.1in} \\ \ds \leq c_{B} \,(t-T^*)\le(I_{T^*,t}(u) +  \int_{T_*}^t |u(s)-z_u(s)|^2\,ds \r),
\end{array}\]
so that, thanks to the
Gronwall lemma, 
\[|u(t)-z_u(t)|^2 \leq c_{T,B}\,I_{T_\star, T}(u).\]
Now, if we fix $\ds r < (\e\,c_{T,B}^{-1})^{1/2}$, then we have $|u-z_u|_{C([\delta,T];\R)} < \epsilon$. Since the constant $c_{T,B}$ is independent of $x$, this proves the result.

Next, we consider the case where $u \in \Phi^{x}_{\delta,T}(s+r)$, but $I_{\delta,T}(u) \leq r$. Let $w \in C([0,\delta];\R)$ be a path such that $w(0) = \langle x,\mu \rangle$, $w(\delta) = u(\delta)$ and $I_{0,\delta}(w) \leq J_{0,\delta}(x,u)+r$. Then similar to before we may define the path $z_u \in C([\delta,T];\R)$ by
\[
z_u(t)= u^{w(T^*)}(t-T^*), \ \ \ \ t \in [\delta,T],
\]
where
\[T^* =T^*(w,r):=\inf \{t \in [0,\delta]:I_{t,\delta}(w) \leq 2r\}.
\]
This implies that
\[I^x_{\delta,T}(z_u) = J_{0,\delta}(x,z_u) \leq I_{0,T^*}(w)  = I_{0,\delta}(w)-I_{T^*,\delta}(w) \leq J_{0,\delta}(x,u)-r \leq I_{\delta,T}^x(u)-r.
\]
Therefore, $z_u \in \Phi^x_{\delta,T}(s)$. Finally, if we consider the path $\tilde{u}(t):= w(t)\mathbb{I}_{[T^*,\delta]}(t) + u(t) \mathbb{I}_{[\delta,T]}(t)$, then $I_{T^*,T}(\tilde{u}) \leq 3r$. Thus by the same calculation as before we obtain 
\[|u-z_u|_C([\delta,T];\R) \leq |\tilde{u}-u^{w(T^*)}(\cdot - T^*)|_{C([T^*,T];\R)} \leq c_{T, B}\,I_{T^*,T}(\tilde{u}),
\]
which completes the proof upon taking $r$ small enough.
\end{proof}

\begin{lemma}
Suppose $x_n \rightharpoonup x$ in $H$. Then for any  $\delta, T > 0$ and $s \geq 0$, we have
\begin{equation*}
\lim_{n \to \infty} \sup_{u \in \Phi_{\delta,T}^x(s)} \mathrm{dist}(u,\Phi_{\delta,T}^{x_n}(s)) = 0,
\end{equation*}
and
\begin{equation*}
\lim_{n \to \infty} \sup_{u \in \Phi_{\delta,T}^{x_n}(s)} \mathrm{dist}(u,\Phi_{\delta,T}^{x}(s)) = 0.
\end{equation*}
In particular, the requirements of Proposition \ref{prop:equiv} are satisfied.
\end{lemma}
\begin{proof}
For fixed $u \in C([0,T];\R)$, the mapping $x \mapsto I^x_{\delta,T}(u)$ is lower semi-continuous. Condition (i) of Proposition \ref{prop:equiv} then follows from Condition (i) in Theorem \ref{theor:laplace} (see the proof of Theorem 5 in \cite{budhiraja2008}).

To show the first limit, it suffices to prove that for any $\{u_n\}_{n=1}^\infty \subset \Phi^x_{\d,T}(s)$ we have 
\begin{equation}
\label{f100}
 \liminf_{n \to \infty} \mathrm{dist}(u_n,\Phi^{x_n}_{\d,T}(s)) = 0.
\end{equation}
Since $I^x_{\d,T}$ is a good rate function, we may assume by taking a subsequence, if necessary, that $u_n \to u \in \Phi^x_{\d,T}(s)$. By \eqref{eq:equiv_cond2}, we may also find a sequence $\{z_n\}_{n=1}^\infty$ such that $z_n \to u$ and 
\[\limsup_{n \to \infty} I^{x_n}_{\d,T}(z_n) \leq s.
\]
Then, for any $r > 0$ we have that
\[\begin{array}{l}
\ds{\mathrm{dist}(u_n,\Phi^{x_n}_{\d,T}(s)) \leq |u_n-z_n|_{C([\delta,T];\R)} + \mathrm{dist}(z_n,\Phi^{x_n}_{\d,T}(s+r)) + \mathrm{dist}(\Phi^{x_n}_{\d,T}(s+ r),\Phi^{x_n}_{\d,T}(s))}\\
\vs
\ds{=|u_n-z_n|_{C([\delta,T];\R)}+\mathrm{dist}(\Phi^{x_n}_{\d,T}(s+ r),\Phi^{x_n}_{\d,T}(s)).}
\end{array}
\]
Therefore, due to the previous lemma, for every $\e>0$ we can find $r_\e>0$ such that 
\[\mathrm{dist}(u_n,\Phi^{x_n}_{\d,T}(s)) \leq |u_n-z_n|_{C([\delta,T];\R)} + \e,\ \ \ \ n\geq 0,\]
and this implies \eqref{f100}.

To show the second limit, it suffices to prove that for any $\{u_n\}_{n=1}^\infty \subset C([\delta,T];\R)$ such that $u_n \in \Phi_{\delta,T}^{x_n}(s)$, we have
\[ \liminf_{n \to \infty} \mathrm{dist}(u_n,\Phi^{x}_{\delta,T}(s)) = 0.
\]
By condition (i), we may assume, by taking a subsequence if necessary, that $u_n \to u$. Then, thanks to Lemma \ref{n10}, we obtain
\[
\liminf_{n \to \infty} I_{\d,T}^{x_n}(u_n)  \geq \liminf_{n \to \infty} J_{0,\d}(x_n,u_n) + \liminf_{n \to \infty} I_{\delta,T}(u_n)  \geq J_{0,\d}(x,u) + I_{\delta,T}(u) = I^x_{\delta,T}(u).
\]
In particular, this implies that $I_{\delta,T}^x(u) \leq s$ so that $u \in \Phi^x_{\delta,T}(s))$. Therefore,
\[\mathrm{dist}_{C([\delta,T];\R)}(u_n,\Phi^{x}_{\delta,T}(s)) \leq |u-u_n|_{C([\delta,T];\R)},
\]
which concludes the proof.
\end{proof}

\begin{remark}
In Proposition \ref{prop:conv}, we have proven that $u_\epsilon^{x_\epsilon,\varphi_\epsilon}$ converges to $u^{x,\varphi}$ in $C([\delta,T];H_\mu)$, $\mathbb{P}$-a.s., for every $0<\d<T$. The reason we do not have convergence (and hence a large deviation principle) in $C([0,T];H_\mu)$ is because $e^{t\frac A\e}x$ does not converge to $\langle x, \mu \rangle$ uniformly on $t \in [0,T]$ as $\epsilon \to 0$. On the other hand, for any $k \geq 1$,  due to $\eqref{eq:longterm}$ we have
\begin{equation*}
\int_0^T |e^{t\frac A\e} x - \langle x, \mu \rangle |_{H_\mu}^k dt \leq c \int_0^T e^{-t\gamma k/\epsilon} |x|_{H_\mu}^k dt \leq c\, \epsilon\, |x|_{H_\mu}^k.
\end{equation*}
This implies that $u_\epsilon^{x_\epsilon,\varphi_\epsilon}$ converges to $u^{x,\varphi}$, as $\epsilon \to 0$, in the space $L^p(\Omega;L^k(0,T;H_\mu))$ for any $p,k \geq 1$. Consequently, the family $\{\mathcal{L}(u_\epsilon^x) \}_{\epsilon \in (0,1]}$ satisfies a large deviation principle in $L^k(0,T;H_\mu)$.
\begin{flushright}
$\Box$
\end{flushright}

\end{remark}

\begin{remark}\label{rem:equiv_LDP}
The action functional $I_{\delta,T}^x$ is  the same action functional that governs the large deviation principle in $C([\delta,T];\R)$ satisfied by the family $\{\mathcal{L}(v_\epsilon^x)\}_{\epsilon > 0}$, where $v_\epsilon^x$ is the solution to the one-dimensional SDE,
\begin{equation}\label{eq:SDE_samelaw}
d v(t) = \bar{F}(v(t))dt + \sqrt{\gamma(\epsilon)\,\mathcal{H}(t,v(t))}\,d\beta(t),\ \ \ \ \ v(0) = \langle x, \mu \rangle.
\end{equation}
The law of the solutions to $\eqref{eq:SDE_samelaw}$ is equal to the law of the solutions to the SDE,
\begin{equation}\label{eq:SDE_onelimit}
du(t) = \bar{F}(t,u(t))dt + \frac{\sqrt{\gamma(\epsilon)}}{1+\bar{\rho}}\,\le( \bar{G}(t,u(t)) dw^Q(t) +\bar{\rho}\,\bar{\Sigma}(t) dw^B(t)\r), \ \ \ \ 
u(0) = \langle x, \mu \rangle.
\end{equation} 
Now, in view of equation $\eqref{eq:SDE_limit}$, we see that $\eqref{eq:SDE_onelimit}$ is precisely the limiting equation of $\eqref{eq:main}$ if the coefficients $\a(\epsilon)$ and $\beta(\e)$ are held fixed, while only the $\epsilon$ terms with the diffusion $\mathcal{A}$ are taken to $0$. Therefore, the large deviation principle would not be affected if we were to take the spatial averaging limit to completion before allowing the noises to decay. 
\begin{flushright}
$\Box$
\end{flushright}
\end{remark}

\section{Applications to the Exit Problem}\label{sec:exit}
In this section we consider the problem of the exit of the process $u_\epsilon^x$, the solution of equation $\eqref{eq:main}$, from a bounded domain $D \subset H_\mu$. With this in mind, we make the following assumptions on the domain $D$ and the coefficients $f,g$ and $\sigma$.

\begin{hyp}\label{hyp:exit}
\begin{enumerate}
\item[(i)]  The coefficients $f,g$ and $\sigma$ are all independent of $t$. In addition,
\begin{equation*}
\sup_{(\xi,r) \in \mathcal{O} \times \R} |g(\xi,r)| < \infty.
\end{equation*}
\item[(ii)] For any $x \in \bar{D}$, the unique solution $u^x$ of the one-dimensional ODE
\begin{equation*}
 \frac{du}{dt} = \bar{F}(u(t)), \ \ \ \ u(0) = \langle x, \mu \rangle,\end{equation*}
 satisfies  $u^x(t) \in \bar{D}$, for any $t \geq 0$. Moreover, 
 for every $c_1, c_2>0$ there exists $T=T(c_1,c_2)>0$ such that
 \[|x|_{H_\mu}\leq c_2 \Longrightarrow |u^x(t)|_{H_\mu}\leq c_1,\ \ \ \ t\geq T.\]

 \item[(iii)] The domain $D \subset H_\mu$ is an open, bounded, connected set that contains $x = 0$. In addition, $D$ is invariant under the semigroup $e^{tA}$ and $\langle x, \mu \rangle \in D$, for each $x \in D$.
\end{enumerate}
\end{hyp}
\begin{remark} 
 The invariance of $D$ under $e^{tA}$ will be necessary in order to prove a lower bound on the exit time of the process $u_\epsilon^x$ from the domain $D$. This is because when $\epsilon$ is small, equation $\eqref{eq:main}$ behaves likes the heat equation,
\begin{equation*}
\frac{\partial u}{\partial t} = \frac{1}{\epsilon}\mathcal{A}u,
\end{equation*}
for $t$ on the order of $\epsilon$. In fact, if $D$ is not invariant under the semigroup $e^{tA}$, then for some $x \in D$ the process $u_\epsilon^x$ will immediately exit the domain, as $\epsilon \to 0$. 
\begin{flushright}
$\Box$
\end{flushright}

\end{remark}

\begin{lemma}
Assume  that $\mathcal{A}$ is a divergence type operator and pick any function $g:\R \to \R$ that is of class $C^2$ and  convex and has   quadratic growth at infinity. 
For every $r \in\,\mathbb{R}$, we define 
\[\mathcal{D}_g(r):=\le\{\,x \in\,H\,:\,G(x)<r\,\r\},\]
where 
\begin{equation*}
G(h) = \int_\mathcal{O} g(h(\xi)) d \xi,\ \ \ \ h \in\,H.
\end{equation*}
Then, there exists $\bar{r} \in\,\mathbb{R}$ such that the domain $\mathcal{D}_g(r)$ satisfies Condition (iii) in Hypothesis \ref{hyp:exit}, for every $r> \bar{r}$.

\end{lemma}

\begin{proof}
First of all, since $g$ has no more that quadratic growth at infinity, the mapping $G:H\to \mathbb{R}$ is well defined. It is differentiable and $G^\prime(h)=g^\prime \circ h$, for every $h \in\,H$. Moreover,  since $\mathcal{A} $ is a divergence type operator, $H = H_\mu$. 

The convexity and the quadratic growth at infinity of $g$ imply, respectively,  that $\mathcal{D}_g(r)$ is convex and bounded, for every $r \in\,\mathbb{R}$. Moreover, $0 \in\,\mathcal{D}_g(r)$, for every $r>g(0)\,|\mathcal{O}|=:\bar{r}.$

Now, we show that $\mathcal{D}_g(r)$ is invariant under the semigroup $e^{tA}$. Actually, if $x \in H$ and $u(t):= e^{tA}x$, by differentiating  and integrating by parts we have
\[\begin{array}{l}
\ds{\frac{d}{dt}\, G(u(t)) = \langle G'(u(t)),\partial_t u(t) \rangle_H = \langle g'(u(t)), \mathcal{A} u(t) \rangle_H}\\
\vspace{.1mm}\\
\ds{ =-\int_\mathcal{O} g''(u(t,\xi)) \langle a(\xi)\nabla u(t,\xi),\nabla u(t,\xi)\rangle\, d\xi\leq 0,}
\end{array}\]
last inequality following from the fact that $g$ is convex and from \eqref{f1}. 
This means that the mapping $t\mapsto G(u(t))$ is non-increasing, so that
\[x \in\,\mathcal{D}_g(r)\Longrightarrow G(e^{tA}x)\leq G(x)<r,\ \ \ t\geq 0.\]

Finally, we show that if $x \in\,\mathcal{D}_g(r)$, then $\langle x,\mu\rangle \in\,\mathcal{D}_g(r)$. We have
\[G(\langle x,\mu\rangle)=\int_{\mathcal{O}}g(\langle x,\mu\rangle)\,d\xi=|\mathcal{O}|\,g(\langle x,\mu\rangle)\leq \int_{\mathcal{O}}g(x(\xi))\,d\xi=G(x)<r.\]

\end{proof}

We have seen that for every $x \in H$ and $\delta > 0$, the family $\{\mathcal{L}(u_\epsilon^x)\}_{\epsilon > 0}$ satisfies a uniform large deviation principle in $C([\delta,T];H_\mu)$ with action functional $I_{\delta,T}^x$. Moreover, if $x$ is constant then $\{\mathcal{L}(u_\epsilon^x)\}_{\epsilon > 0}$ satisfies a large deviation principle in $C([0,T];H_\mu)$ with action functional $I_{0,T}^x$. On the basis of this, we define the quasipotential $V:H_\mu \to [0,+\infty]$, by
\begin{equation*}
V(y):=\inf\{\,I_{0,T}^0(u)\ :\ u \in C([0,T];H_\mu),\  u(T) = y,\ T>0\,\}.
\end{equation*}
Recalling that $I_{0,T}^x$ is finite only if $u \in C([0,T];\R)$, it follows that
\begin{equation*}
V(y) < +\infty \implies y \text{ is constant}.
\end{equation*} 
Moreover, since we assume that $D$ contains a ball around $0$, it follows that both $D$ and $\partial D$ will contain some constant $y \in H_\mu$. In particular, there will exist paths starting at $0$ and ending at $z\in \partial D$ that only travel along the subspace $\{y \in H_\mu:y \text{ is constant} \}$. These paths will have finite values of the action functional, so that
\begin{equation}\label{eq:Vbar}
\bar{V}(D):= \inf_{y \in \partial D} V(y) < + \infty.
\end{equation}
In addition, due to Condition (ii) of Hypothesis \ref{hyp:exit}, the intersection of $D$ and the subspace $\R \subset H_\mu$ is precisely an open interval containing $0$. Therefore, if we denote the endpoints of the interval $\R \cap D$ by $y_1$ and $y_2$, then $\bar{V}(D) = \min(V(y_1),V(y_2))$.
\begin{remark}
Suppose $g \equiv 1$, so that the noise is additive. As discussed in Remark \ref{rem:equiv_LDP}, $I_{0,T}^x$ is the action functional for the large deviation principle satisfied by the family $\{\mathcal{L}(v_\epsilon^x)\}$, where $v_\epsilon^x$ is the solution of
\begin{equation*}
d v(t) = \bar{F}(v(t))dt + \sqrt{\gamma(\epsilon)\,\mathcal{H}_{\bar{\rho}}}\, d\beta(t),\ \ \ \  v(0) = \langle x, \mu \rangle,
\end{equation*}
with 
\[\mathcal{H}_{\bar{\rho}}=\frac 1{(1+\bar{\rho})^2}\le( |\sqrt{Q}m|_H^2 + \bar{\rho}^2\,\delta_0^2 |\sqrt{B}[\Sigma N_{\delta_0}^*m]|_Z^2\r).\] Therefore, due to classical results (see \cite{freidlin1998random}), we will have the explicit formula,
\begin{equation*}
V(y) = -\frac{2}{\mathcal{H}_{\bar{\rho}}}\, \int_0^y \bar{F}(\sigma)d\sigma.
\end{equation*}
In the case that the noise is multiplicative, there is no such explicit representation of the quasipotential, but the exit results we discuss below will still hold.
\begin{flushright}
$\Box$
\end{flushright}

\end{remark}

Our goal in this section is to a prove Freidlin-Wentzell type estimates on the exit time of $u_\epsilon^x$ from the domain $D$. With this in mind, we define the stopping times,
\begin{equation*}
\tau_\epsilon^x := \inf \{t \geq 0: u_\epsilon^x(t) \in \partial D \}.
\end{equation*}
The main result is the following.
\begin{theorem}\label{theor:exit}
Assume that all  Hypotheses \ref{hyp:longterm} to \ref{hyp:exit} are satisfied. Then for any $x \in D$, we have
\begin{equation*}
\lim_{\epsilon \to 0} \epsilon \log \mathbb{E} \tau^x_\epsilon = \bar{V}(D).
\end{equation*}
\end{theorem}

The proof of Theorem \ref{theor:exit} is a consequence of  the following series of lemmas. Once these lemmas are established, the proof of Theorem \ref{theor:exit}  proceeds as in the finite dimensional case (see Theorem 5.7.11 of \cite{dembo2009large}). We list the lemmas below, and postpone their proofs until Appendix \ref{sec:exit_lemmas}.

In what follows, we set
\begin{equation*}
B_\rho := \{y \in H_\mu:|y|_{H_\mu} \leq \rho\},
\end{equation*} 
and, for every $\rho>0$ such that $B_\rho\subset D$,  we define the stopping times
\begin{equation*}
\sigma_\epsilon^x(\rho):=\inf \{t \geq 0: u_\epsilon^x(t) \in B_\rho \cup \partial D\}.
\end{equation*}
\begin{lemma}\label{lem:18}
For any $\eta > 0$, there exists a $T < \infty$ such that
\begin{equation*}
\liminf_{\epsilon \to 0} \gamma(\epsilon) \log \inf_{x \in D} \mathbb{P}_x(\tau_\epsilon^x \leq T) > -(\bar{V}(D) + \eta).
\end{equation*}
\end{lemma} 
\begin{lemma}\label{lem:19}
Let $\rho>0$ be such that $B_\rho \subset D$. Then
\begin{equation*}
\lim_{t \to \infty} \limsup_{\epsilon \to 0} \gamma(\epsilon) \log \sup_{x \in D} \mathbb{P}(\sigma_\epsilon^x(\rho) > t) = - \infty.
\end{equation*}
\end{lemma}
\begin{lemma} \label{lem:22}
Let $\rho >0$ be such that $B_\rho \subset D$. Then, for any $x \in D$,
\begin{equation*}
\lim_{\epsilon \to 0} \mathbb{P}\le(u^x_\epsilon(\sigma_\epsilon^x(\rho)) \in B_\rho \r) = 1.
\end{equation*}
\end{lemma}
\begin{lemma}\label{lem:23}
Let $\rho >0$ be such that $B_\rho \subset D$. Then for any $\eta > 0$, there exists $T< \infty$ such that
\begin{equation*}
\limsup_{\epsilon \to 0} \gamma(\epsilon) \log \sup_{x \in B_\rho} \mathbb{P}(\sup_{0 \leq t \leq T} |u_\epsilon^x(t) - x|_{H_\mu}\geq 3\rho) < -\eta.
\end{equation*}
\end{lemma}
\begin{lemma} \label{lem:21}
Let $\rho >0$ be such that $B_{2\rho} \subset D$. Then, for any closed set $N \subset \partial D$, we have
\begin{equation*}
\lim_{\rho \to 0} \limsup_{\epsilon \to 0} \gamma(\epsilon) \log \sup_{x \in \partial B_{2 \rho}}\mathbb{P}\le(u^x_\epsilon(\sigma_\epsilon^x(\rho)) \in N \r) \leq -\inf_{z \in N} V(z).
\end{equation*}
\end{lemma}

\section{Appendix A: Some Lemmas used in Section \ref{sec:proof}}
We start with a first preliminary result.

\begin{lemma}
For every $\varphi, \psi \in H$  and  $t>0$ we have
\begin{equation}\label{eq:Glemma_3}
|e^{tA}\,(\varphi \sqrt{Q} \psi)|_{H_\mu} \leq c \le(t^{-\frac{d}{2\zeta}}+1\r) |\varphi|_{H_\mu} |\psi|_H.
\end{equation}
\end{lemma}
\begin{proof}
If we set $\zeta = \frac{2\rho}{\rho-2}$, we have $\frac{1}{\zeta}+\frac{1}{\rho}+\frac{1}{2} = 1$. Thus, for any $t > 0$ and $\psi \in H$, due to condition \eqref{eq:rhohyp} we have
\begin{equation}\label{eq:Glemma_1}
\begin{array}{l}
\ds{ |e^{tA}(\varphi \sqrt{Q} \psi)|_{H_\mu}  = \le| \sum_k \lambda_k  \langle \psi, e_k \rangle e^{tA}(\varphi e_k) \r|_{H_\mu}}\\
\vspace{.1mm}\\
\ds{  \leq \le(\sum_k \lambda_k^{\rho} |e_k|_\infty^{2} \r)^{\frac 1\rho}  \le(\sum_k |e_k|_\infty^{-\frac{2\zeta}\rho} \big|e^{tA}(\varphi e_k) \big|_{H_\mu}^\zeta \r)^{\frac 1\zeta} |\varphi|_H}\\
\vspace{.1mm}\\
\ds{ \leq \kappa_Q^{\frac 1\rho} |\psi|_{H} \le(\sum_k \big|e^{tA}(\varphi e_k) \big|_{H_\mu}^2 \r)^{\frac 1\zeta} \sup_k \le( |e_k|_\infty^{-\frac 2\rho} \le|e^{tA}(\varphi e_k)\r|_{H_\mu}^{\frac{\zeta-2}\zeta} \r).}
\end{array}
\end{equation}
By Remark \ref{rem:embedding}, the semigroup is a contraction on $H_\mu$.  Then, since $(\zeta-2)/\zeta = 2/\rho$,
\begin{equation}\label{eq:Glemma_2}
\sup_{k \geq 0}\,  |e_k|_\infty^{-\frac 2\rho}  \le|e^{tA}(\varphi e_k)\r|_{H_\mu}^{\frac{\zeta-2}\zeta}\leq   \sup_{k \geq 0}\,  |e_k|_\infty^{-\frac 2\rho}  \le|\varphi e_k\r|_{H_\mu}^{\frac{2}\rho}\leq |\varphi|_{H_\mu}^{\frac 2\rho}.
\end{equation}
Moreover, thanks to $\eqref{eq:sem_kernel}$ and the invariance of the semigroup with respect to the measure $\mu$, we obtain
\[\begin{array}{l}
\ds{\sum_k \le|e^{tA}(\varphi e_k) \r|_{H_\mu}^2   = \int_\mathcal{O} \sum_k |\langle k_t(\xi,\cdot)\varphi(\cdot),e_k(\cdot) \rangle|^2 d\mu(\xi) = \int_\mathcal{O} |k_t(\xi,\cdot)\varphi(\cdot)|_H^2 d\mu(\xi)}\\
\vspace{.1mm}\\
\ds{  \leq c\,(t^{-\frac d2}+1) \int_\mathcal{O} e^{tA} |\varphi(\xi)|^2d\mu(\xi) = c\, (t^{-\frac d2}+1) |\varphi|_{H_\mu}^2.}
\end{array}\]
Due to $\eqref{eq:Glemma_1}$ and $\eqref{eq:Glemma_2}$, this implies that \eqref{eq:Glemma_3} holds.
\end{proof}

Now, we are ready to state and  prove all lemmas used in Section \ref{sec:proof}.
\begin{lemma}\label{lem:Fterm}
For every $\e>0$, let us define 
\begin{equation*}
I^1_\epsilon(t)= \int_0^t e^{(t-s)\frac A\e} F(s,u_\epsilon^{x_\epsilon,\varphi_\epsilon}(s))ds  -\int_0^t \bar{F}(s,u^{x,\varphi}(s)) ds,
\end{equation*}
as in Proposition \ref{prop:conv}. Then for any $p \geq 1$ and $T > 0$, we have
\begin{equation}\label{eq:Flemma_main}
\mathbb{E}\sup_{0 \leq t \leq T} |I^1_\epsilon(t)|_{H_\mu}^p \leq r_{T,p}(\e)+ c_{T, p}\, \int_0^T \mathbb{E}\sup_{0 \leq s \leq t} |u_\epsilon^{x_\epsilon,\varphi_\epsilon}(s)-u^{x,\varphi}(s)|_{H_\mu}^p\,ds,
\end{equation}
where $r_{T,p}(\epsilon)$ is some non-negative function such that $r_{T,p}(\epsilon) \to 0$, as $\epsilon \to 0$.
\end{lemma}
\begin{proof}
We can rewrite $I^1_\epsilon(t)$ as follows.
\[\begin{array}{l}
\ds{I^1_\epsilon(t) = J_1^\epsilon(t) + J_2^\epsilon(t):= \le(\int_0^t e^{(t-s)\frac A\e}F(s,u_\epsilon^{x_\epsilon,\varphi_\epsilon}(s))ds - \int_0^t\bar{F}(s,u_\epsilon^{x_\epsilon,\varphi_\epsilon}(s)) ds   \r) }\\
\vspace{.1mm}\\
\ds{ + \le( \int_0^t \le[ \bar{F}(s,u_\epsilon^{x_\epsilon,\varphi_\epsilon}(s)) - \bar{F}(s,u^{x,\varphi}(s)) \r] ds \r).}
\end{array}\]
Concerning $J_1^\epsilon$, thanks to $\eqref{eq:longterm}$ and the a priori estimate $\eqref{eq:apriori}$, we obtain
\begin{equation}\label{eq:Flemma_1}
\begin{array}{l}
\ds{
 \mathbb{E} \sup_{0 \leq t \leq T} |J_1^\epsilon(t)|_{H_\mu}^p  \leq  c\,\mathbb{E} \sup_{0 \leq t \leq T} \le(\int_0^t  e^{-\frac{\gamma(t-s)}\epsilon} |F(s,u_\epsilon^{x_\epsilon,\varphi_\epsilon}(s))|_{H_\mu}ds \r)^p }\\
\vspace{.1mm}\\
\ds{  \leq c_p \le( 1+\mathbb{E}  \sup_{0 \leq t \leq T}|u_\epsilon^{x_\epsilon,\varphi_\epsilon}(t)|_{H_\mu}^p \r) \le(\int_0^T e^{-\frac{\gamma s}\epsilon} ds \r)^p  \leq \epsilon^p c_{T,p,M} (1+|x_\epsilon|_H^p) \leq \epsilon^p c_{T,p,M},}
\end{array}
\end{equation}
where the last inequality follows from the fact that the sequence $\{x_\epsilon\}_{\epsilon>0}$ is weakly convergent and hence resides in a bounded set of $H$. Next, concerning $J_2^\epsilon(t)$, we have
\[\begin{array}{l}
\ds{\mathbb{E}\sup_{0 \leq t \leq T}|J_2^\epsilon(t)|^p  \leq  T^{p-1} \mathbb{E} \int_0^T  |\bar{F}(s,u_\epsilon^{x_\epsilon,\varphi_\epsilon}(s)) - \bar{F}(s,u^{x,\varphi}(s))|^p ds}\\
\vspace{.1mm}\\
\ds{ \leq c\, T^{p-1}  \int_0^T \mathbb{E}\sup_{0 \leq s \leq t}|u_\epsilon^{x_\epsilon,\varphi_\epsilon}(s) - u^{x,\varphi}(s)|_{H_\mu}^p dt.}
\end{array}\]
This inequality, together with $\eqref{eq:Flemma_1}$, implies \eqref{eq:Flemma_main}.
\end{proof}

\begin{lemma}\label{lem:Gterm}
For every $\e>0$, let us define
\begin{equation*}
I^2_\epsilon(t)= \int_0^t e^{(t-s)\frac A\e} G(s,u_\epsilon^{x_\epsilon,\varphi_\epsilon}(s))\le[\sqrt{Q}\varphi_H^\epsilon(s)\r]ds - \int_0^t \bar{G}(s,u^{x,\varphi}(s)) \le[\sqrt{Q}\varphi_H(s) \r]ds,
\end{equation*}
 as in Proposition \ref{prop:conv}. Then, for every $p\geq 1$ and $T>0$, the following estimate holds.
\begin{equation*}
\mathbb{E}\sup_{0 \leq t \leq T} |I^2_\epsilon(t)|_{H_\mu}^2 \leq r_{T,p}(\e)+c_{T,p} \, \int_0^T \mathbb{E}\sup_{0 \leq s \leq t} |u_\epsilon^{x_\epsilon,\varphi_\epsilon}(s)-u^{x,\varphi}(s)|_{H_\mu}^2 \,dt,
\end{equation*}
where $r_{T,p}(\epsilon)$ is some non-negative function such that $r_{T,p}(\epsilon) \to 0$, as $\epsilon \to 0$.
\end{lemma}
\begin{proof}
We can rewrite $I^2_\epsilon(t)$ as follows.
\[\begin{array}{l}
\ds{I^2_\epsilon(t)  = \le(\int_0^t e^{(t-s)\frac A\e} G(s,u_\epsilon^{x_\epsilon,\varphi_\epsilon}(s))\le[\sqrt{Q}\varphi_H^\epsilon(s)\r]\,ds - \int_0^t \bar{G}(s,u^{x_\e,\varphi_\e}(s))\le[\sqrt{Q}\varphi_H^\epsilon(s)\r]\, ds   \r)}\\
\vspace{.1mm}\\
\ds{ + \le( \int_0^t \bar{G}(s,u^{x,\varphi}(s))\le[ \sqrt{Q}(\varphi_H^\epsilon(s)-\varphi_H(s))\r]\, ds   \r) }\\
\vspace{.1mm}\\
\ds{ + \le( \int_0^t \le(\bar{G}(s,u_\epsilon^{x_\epsilon,\varphi_\epsilon}(s))-\bar{G}(s,u^{x,\varphi}(s)) \r])\,\le[\sqrt{Q}\varphi_H^\epsilon(s)\r]\, ds   \r)   =: \sum_{i=1}^3 J_i^\epsilon(t).}
\end{array}\] 

{\em Step 1.} We first show that for any $p \geq 1$, \begin{equation}\label{eq:Glemma_4}
\lim_{\epsilon \to 0}\mathbb{E} \sup_{0 \leq t \leq T} \le| J_1^\epsilon(t)\r|_{H_\mu}^p =0.
\end{equation}
Due to the invariance of the semigroup with respect to $\mu$ and $\eqref{eq:longterm}$, we have 
\[\begin{array}{l}
\ds{|J_1^\epsilon(t)|_{H_\mu}  \leq \le| \int_0^t  e^{(t-s)\frac A\e}G(s,u_\epsilon^{x_\epsilon,\varphi_\epsilon}(s))\le[\sqrt{Q}\varphi^\epsilon_H\r] ds\r.}\\
\vspace{.1mm}\\
\ds{ \le. - \int_0^t \langle e^{(t-s)\frac A{2\e}} G(s,u_\epsilon^{x_\epsilon,\varphi_\epsilon}(s))\le[ \sqrt{Q}\varphi^\epsilon_H(s)\r], \mu \rangle ds \r|_{H_\mu} }\\
\vspace{.1mm}\\
\ds{ \leq  \int_0^t e^{-\frac{\gamma (t-s)}{2\epsilon}} \le|e^{(t-s)\frac A{2\e}}G(u_\epsilon^{x_\epsilon,\varphi_\epsilon})\le[\sqrt{Q}\varphi^\epsilon_H(s)\r]\r|_{H_\mu} ds.}
\end{array}\]
Note that $\frac{d}{\zeta} < 1$ since $ \rho <\frac{2d}{d-2}$. Then, by applying  inequality \eqref{eq:Glemma_3} with $\theta = g(s,\cdot,u_\epsilon^{x_\epsilon,\varphi_\epsilon}(s,\cdot))$ we conclude that
\[\begin{array}{l}
\ds{\le| J_1^\epsilon(t) \r|_{H_\mu}   \leq c\int_0^t e^{-\frac{\gamma(t-s)}{2\epsilon}}\le[\le((t-s)/\e\r)^{-\frac d{2\zeta}} +1 \r] |g(s,\cdot,u_\epsilon^{x_\epsilon,\varphi_\epsilon}(s))|_{H_\mu} |\varphi_H^\epsilon(s)|_H\, ds}\\
\vspace{.1mm}\\
\ds{  \leq c\, \le(\int_0^T e^{-\frac {\gamma t}\epsilon}\le[\le(t/\e \r)^{-\frac d\zeta}+1\r]ds \r)^{1/2} \le(\int_0^T |\varphi_H^\epsilon(s)|^2_H ds \r)^{1/2}\le(1 + \sup_{0 \leq s \leq t}|u_\epsilon^{x_\epsilon,\varphi_\epsilon}(s)|_{H_\mu}  \r)}\\
\vspace{.1mm}\\
\ds{ \leq c_M\epsilon^{\frac 12} \le(1 + \sup_{0 \leq s \leq t}|u_\epsilon^{x_\epsilon,\varphi_\epsilon}(s)|_{H_\mu} \r).}
\end{array}\]
In view of estimate $\eqref{eq:apriori}$, since $\sup_{\epsilon \in (0,1]}|x_\epsilon|_{H_\mu} < \infty$, we obtain $\eqref{eq:Glemma_4}$ upon taking the $p$th moment. 

\medskip

{\em Step 2.} We show that for any $p \geq 1$  \begin{equation}\label{eq:Glemma_5}
\lim_{\epsilon \to 0}\,    \mathbb{E} \sup_{0 \leq t \leq T} \le| J_2^\e(t)\r|^p =0.
\end{equation}
For every $\psi \in\,L^2(0,T;H)$, we define
 \begin{equation*}
\Lambda_\psi(t):= \int_0^t \bar{G}(s,u^{x,\varphi}(s)) \le[\sqrt{Q}\psi(s)\r]\, ds,\ \ \ \ t \in\,[0,T].
 \end{equation*}
 First, we show that the family $\{\Lambda_{\varphi^\epsilon_H}\}_{\epsilon \in (0,1]}$ is equi-continuous and equi-bounded in $[0,T]$, $\mathbb{P}$-a.s. Actually, we have
\[\begin{array}{l}
\ds{|\Lambda_{\varphi_H^\epsilon}(t+h)-\Lambda_{\varphi_H^\epsilon}(t)|  =  \le|\int_t^{t+h}  \bar{G}(s,u^{x,\varphi}(s))\le[\sqrt{Q}\varphi_H^\epsilon(s)\r]\,ds \r| }\\
\vspace{.1mm}\\
\ds{ \leq c \le(\int_t^{t+h} \int_\mathcal{O} |g(s,\xi,u^{x,\varphi}(s))|^2 d\mu(\xi) ds \r)^{1/2} \le(\int_t^{t+h}\int_\mathcal{O} |\varphi^\epsilon_H(s,\xi)|^2d\mu(\xi)ds\r)^{1/2}}\\
\vspace{.1mm}\\
\ds{ \leq c \le(\int_t^{t+h} (1+|u^{x,\varphi}(s)|^2)ds \r)^{1/2} |\varphi^\epsilon_H |_{L^2(0,T;H)}.}
\end{array}\]
Then, since $u^{x,\varphi} \in C([0,T];\R)$, $\mathbb{P}$-a.s., we have that
\begin{equation}\label{eq:Glemma_6}
\sup_{\epsilon \in (0,1]} \sup_{0 \leq t \leq T} |\Lambda_{\varphi_H^\epsilon}(t+h)-\Lambda_{\varphi_H^\epsilon}(t)| \leq c_M \sqrt{h}, \qquad \mathbb{P}-a.s.,
\end{equation}
for some random variable $c_M \in\,L^2(\Omega)$. Next, we observe that for each fixed $t \in [0,T]$ the linear functional $\psi \in\,L^2(0,T;H) \mapsto \Lambda_\psi(t) \in\,\mathbb{R}$ is  bounded. Therefore, by the weak convergence of the sequence $\{\varphi^\epsilon_H\}$ to $\varphi_H$, we may conclude that
\begin{equation*}
\lim_{\epsilon \to 0} \Lambda_{\varphi_H^\epsilon}(t) = \Lambda_{\varphi_H}(t)=\int_0^t \bar{G}(s,u^{x,\varphi}(s))\le[\sqrt{Q}\varphi_H(s) \r]ds, \ \ \ \ \mathbb{P}-a.s.,
\end{equation*}
 and estimate $\eqref{eq:Glemma_6}$ implies that this convergence is uniform with respect to $t \in [0,T]$. Finally, noting that $J_2^\epsilon(t) = \Lambda_{\varphi_H^\epsilon}(t)-\Lambda_{\varphi_H}(t)$, we conclude that $\eqref{eq:Glemma_5}$ holds from the dominated convergence theorem.

\medskip

{\em Step 3.}
Using the Lipschitz continuity of $g$, we have
\[\begin{array}{l}
\ds{|J_3^\epsilon(t)|^2  \leq \le(\int_0^t\int_\mathcal{O} \le|\le(G(s,u_\epsilon^{x_\epsilon,\varphi_\epsilon})-G(s,u^{x,\varphi})\r)\le[\sqrt{Q}\varphi^\epsilon_H(s)\r]\, \r| d\mu(\xi)ds \r)^2}\\
\vspace{.1mm}\\
\ds{ \leq c\,|u_\epsilon^{x_\epsilon,\varphi_\epsilon}-u^{x,\varphi}|_{L^2(0,T;H_\mu)}^2 |\varphi_H^\epsilon|_{L^2(0,T;H_\mu)}^2 \leq c_M\int_0^T \sup_{0 \leq s \leq t} |u_\epsilon^{x_\epsilon,\varphi_\epsilon}(s)-u^{x,\varphi}(s)|_{H_\mu}^2 dt.}
\end{array}\]
This, together with $\eqref{eq:Glemma_4}$ and $\eqref{eq:Glemma_5}$, concludes the proof.
\end{proof}

\begin{lemma}\label{lem:Sterm} For every $\e>0$, let us define 
\begin{equation*}
I^3_\epsilon(t) = (\delta_0 - A) \int_0^t e^{(t-s)\frac A\e} N_{\delta_0} \le[\Sigma(s) \sqrt{B}\varphi_Z^\epsilon(s) \r]ds - \delta_0\int_0^t \langle  N_{\delta_0}\le[\Sigma(s)\sqrt{B}\varphi_Z(s)\r], \mu \rangle ds,
\end{equation*}
as in Proposition \ref{prop:conv}. Then for any $p \geq 1$,
\begin{equation*}
\lim_{\epsilon \to 0}\mathbb{E} \sup_{0 \leq t \leq T} |I^3_\epsilon(t)|_{H_\mu}^p = 0.
\end{equation*}
\end{lemma}
\begin{proof}
We can rewrite $I^3_\epsilon$ as follows.
\[\begin{array}{l}
\ds{I^3_\epsilon(t)  = \le( (\delta_0 - A) \int_0^t e^{(t-s)\frac A\e} N_{\delta_0}[\Sigma(s)\sqrt{B}\varphi^\epsilon_Z(s)]ds - \delta_0 \int_0^t \langle N_{\delta_0}[\Sigma(s)\sqrt{B}\varphi^\epsilon_Z(s)],\mu \rangle ds  \r)}\\
\vspace{.1mm}\\
\ds{ +\delta_0 \int_0^t \langle N_{\delta_0}[\Sigma(s)\sqrt{B}(\varphi^\epsilon_Z(s) - \varphi_Z(s))],\mu \rangle ds  =: J_1^\epsilon(t)+J_2^\epsilon(t).}
\end{array}\]
 Concerning $J_1^\epsilon$, the invariance of $\mu$ gives  
\[\begin{array}{l}
\ds{ |J_1^\epsilon(t)|_{H_\mu}  =  \le| \int_0^te^{(t-s) \frac A{2\e}}  (\delta_0 - A) e^{(t-s)\frac A{2\e}} N_{\delta_0}[\Sigma(s)\sqrt{B}\varphi^\epsilon_Z(s)]ds\r. }\\
\vspace{.1mm}\\
\ds{ \le. - \int_0^t \langle (\delta_0 - A) e^{(t-s) \frac A{2\e}} N_{\delta_0}[\Sigma(s)\sqrt{B}\varphi^\epsilon_Z(s)],\mu \rangle ds  \r|_{H_\mu}
}\\
\vspace{.1mm}\\
\ds{   \leq c \int_0^t e^{-\frac{\gamma (t-s)}{2\epsilon}} \le| (\delta_0-A)e^{(t-s)\frac A{2\e}}N_{\delta_0}[\Sigma(s)\sqrt{B} \varphi^\epsilon_Z(s)] \r|_{H_\mu} ds.}
\end{array}\]
Then, thanks to $\eqref{eq:Srho_identity}$ and the boundedness of the operator $S_\rho$ defined in \eqref{sro},  for any $\rho > 0$, we have
\[\begin{array}{l}
\ds{\le| (\delta_0-A) e^{(t-s) \frac A{2\e}}N_{\delta_0}[\Sigma(s)\sqrt{B} \varphi^\epsilon_Z(s)] \r|_{H_\mu}  \leq c \le[1 + \le((t-s)/\e \r)^{-\frac{1+\rho}4} \r] |\varphi^\epsilon_Z(s)|_Z.}
\end{array}\]
Hence, if $\rho < 1$, we obtain 
\begin{equation}\label{eq:Slemma_1}
 \mathbb{E} \sup_{0 \leq t \leq T}|J_1^\epsilon(t)|_{H_\mu}^p  \leq c \le( \int_0^T e^{-\frac {\gamma t}\epsilon}\le[\le(t/\e
\r)^{-\frac{1+\rho}2}+1\r] ds\r)^{\frac p2} \mathbb{E}\,  |\varphi^\epsilon_Z|_{L^2(0,T;Z)} ^{\frac p2}\leq c_M\, \epsilon^{\frac p2}.
\end{equation}
To estimate $J_2^\e(t)$, we proceed as in Lemma \ref{lem:Gterm} and for every $\psi \in\,L^2(0,T;Z)$ we define
\begin{equation*}
\Lambda_{\psi}(t):=\delta_0 \int_0^t \langle N_{\delta_0} [\Sigma(s)\sqrt{B} \psi(s)], \mu \rangle ds.
\end{equation*}
Then the family $\{\Lambda_{\varphi^\epsilon_Z}\}_{\epsilon \in (0,1]}$ is uniformly equi-continuous in $[0,T]$, since 
\[\begin{array}{l}
\ds{|\Lambda_{\varphi^\epsilon_Z}(t+h)-\Lambda_{\varphi^\epsilon_Z}(t)|  = \le| \delta_0 \int_t^{t+h} \langle N_{\delta_0}[\Sigma(s) \sqrt{B} \varphi^\epsilon_Z(s)],\mu \rangle ds\r| 
}\\
\vspace{.1mm}\\
\ds{ \leq \delta_0 \sqrt{h} \le(\int_t^{t+h} \int_\mathcal{O} |N_{\delta_0}[\Sigma(s)\sqrt{B} \varphi^\epsilon_Z(s)](\xi)|^2 d\mu(\xi) \  ds \r)^{1/2}}\\
\vspace{.1mm}\\
\ds{  \leq c\, \delta_0\, \sqrt{h} | N_{\delta_0}[\Sigma(\cdot)\varphi^\epsilon_Z]|_{L^2(0,T;H_\mu)}  \leq  c_M \sqrt{h},}
\end{array}\]
where the last inequality holds $\mathbb{P}$-a.s., for some random variable $c_M \in\,L^1(\Omega)$. In addition, for fixed $t\in[0,T]$, the linear functional $\psi \in\,L^2(0,T;Z) \mapsto \Lambda\psi(t) \in\,\mathbb{R}$ is  bounded. Hence by the weak convergence of the sequence $\{\varphi^\epsilon_Z\}$ to $\varphi_Z$, we have 
\begin{equation*}
\lim_{\epsilon \to 0} \Lambda_{\varphi^\epsilon_Z}(t) = \Lambda_{\varphi_Z}(t)= \delta_0 \int_0^t\langle N_{\delta_0}[\Sigma(s)\sqrt{B} \varphi_Z(s)] ,\mu \rangle ds.
\end{equation*}
Moreover, this convergence is uniform in $t \in\,[0,T]$, so that
\begin{equation*}
\lim_{\epsilon \to 0} \mathbb{E} \sup_{0 \leq t \leq T} |J_2^\epsilon(t)|^p = \lim_{\epsilon \to 0} \sup_{0 \leq t \leq T}|\Lambda_{\varphi^\epsilon_Z}(t)-\Lambda_{\varphi_Z}(t)|^p = 0
\end{equation*}
from the dominated convergence theorem. This, together with $\eqref{eq:Slemma_1}$, concludes the proof.
\end{proof}

\section{Appendix B: Proofs of Lemmas in Section \ref{sec:exit}}\label{sec:exit_lemmas}

\begin{proof}[Proof of Lemma \ref{lem:18}]
Fix $\eta > 0$. We first construct a collection of paths $\{z^x\}_{x \in D} \subset C([0,T];\R)$ that leave the domain with a close to minimal energy. 

Let $\rho > 0$ such that $B_\rho\subset D$. Due to Condition  (ii) in Hypothesis \ref{hyp:exit}, we can fix $T_1$ large enough that $u^x(T_1) \in B_\rho$, for any $x \in \bar{D}$, where $u^x$ is the solution of $\eqref{eq:ODE}$. Thus, we  set $z^x(t) = u^x(t)$ on the interval $[0,T_1]$. Next, we set
\begin{equation*}
z^x(t) = z^x(T_1)(T_1+1-t) , \qquad \text{ if } t \in [T_1,T_1+1],
\end{equation*}
so that $z^x(T_1+1) = 0$. Now, due to $\eqref{eq:Vbar}$, there exists some $T_2 >0$ and some path $v(t) \in C([0,T_2],\R)$ such that $v(0) = 0$, $v(T_2) \notin \bar{D}$ and $\ds I_{0,T_2}^0(v) < \bar{V}(D)+\eta/4$. We then set $z^x(T_1+1+t) = v(t)$ for $t \in [0,T_2]$. Hence, upon defining $T^* = T_1+T_2+1$, we have
\[\begin{array}{l}
\ds{I_{0,T^*}^x(z^x)  = I_{0,T_1}^x(z^x)+I_{T_1,T_1+1}(z^x)+I_{T_1+1,T^*}^0(z^x)}\\
\vspace{.1mm}\\
\ds{ \leq c\, \int_{T_1}^{T_1+1}|{{z}^x}^\prime(t)-\bar{F}(z^x(t))|^2dt + (\bar{V}(D)+\eta/4)\leq c\,\rho^2+\le(\bar{V}(D)+\eta/4\r).}
\end{array}\]
Thus,  taking $\rho$ small enough, we obtain $I_{0,T}^x(z^x) < \bar{V}(D) + \eta/2$. We note that all of these paths $\{z^x\}_{x \in D}$ agree on the interval $[T_1+1,T^*]$ and exit the domain on this time interval. Let us now denote
\begin{equation*}
h := \sup_{T_1+1 \leq t \leq T^*} \mathrm{dist}_{H_\mu}(z^x(t),\bar{D}) > 0.
\end{equation*}
To prove the lemma, we pick any $0<\delta < T_1+1$ and define the open set
\begin{equation*}
\Psi = \bigcup_{x \in D} \le\{u \in C([\delta,T^*];H_\mu): \sup_{\delta \leq t \leq T^*}|u(t)-z^x(t)|_{H_\mu} < h \r\}.
\end{equation*}
Then, thanks to Theorem \ref{theor:main} and bound $\eqref{eq:LDP_lower}$, there exists  $\epsilon_0>0$ such that for any $\epsilon < \epsilon_0$,
\[\begin{array}{l}
\ds{\inf_{x \in D} \mathbb{P}(\tau_x^\epsilon < T^*)  \geq 
\inf_{x \in D} \mathbb{P}(u^x_\epsilon \in \Psi)  
\geq \exp \le( -\frac{1}{\gamma(\epsilon)}\le[\,\sup_{x \in D} \inf_{\varphi \in \Psi} I_{\delta,T^*}^x(\varphi)+ \frac{\eta}{2} \r] \r)}\\
\vspace{.1mm}\\
\ds{ \geq \exp \le(-\frac{1}{\gamma(\epsilon)}\le[\,\sup_{x \in D}
 I_{\delta,T^*}^x(z^x|_{[\delta,T^*]}) + \frac{\eta}{2}  \r] \r\} \geq \exp \le\{-\frac{1}{\gamma(\epsilon)}(\bar{V}(D)+\eta) \r).}
\end{array}\]
\end{proof}

\begin{proof}[Proof of Lemma \ref{lem:19}]
In this lemma, the behavior of the process near $t = 0$ is not a concern and so the same proof as in  Lemma 5.7.19 in \cite{dembo2009large} holds.
\end{proof}

\begin{proof}[Proof of Lemma \ref{lem:22}]
Fix some $x \in D$ and let $\rho>0$ be such that $B_\rho \subset D$. If $x \in\,B_\rho$, there nothing to prove. Thus, we can assume that $x\notin B_\rho$. 

We denote $T_x:= \inf \{t \geq 0: u^x(t) \in B_{\rho/2}\}$ and $\Delta_x := \inf_{t \geq 0} \mathrm{dist}_{H_\mu}(u^x(t),\partial D)$.
We clearly have $T_x>0$ and, due to Condition (iii) of Hypothesis \ref{hyp:exit}, we have $\Delta_x>0$. Moreover,  again thanks to Condition (iii) of Hypothesis \ref{hyp:exit}, we have
\begin{equation*}
d_x:=\inf_{t \geq 0} \mathrm{dist}_{H_\mu}(e^{tA}x,\partial D) > 0.
\end{equation*}
 This implies that for every $0<\d<T_x$
\begin{equation}
\label{eq:lem22_1}
\begin{array}{l}
\ds{ \mathbb{P}\le(u_\epsilon^x(\sigma_\epsilon^x(\rho)) \in \partial D\r)}\\
 \vspace{.1mm}\\
 \ds{  \leq \mathbb{P} \le(\sup_{0 \leq t \leq \delta} |u_\epsilon^x(t) -e^{t \frac A\e}x|_{H_\mu} > d_x \r)  + \mathbb{P}\le(\sup_{\delta \leq t \leq T_x} |u_\epsilon^x(t)-u^x(t)|_{H_\mu} > \Delta_x\wedge \rho/2\r).}
 \end{array}
\end{equation}
Now, thanks to  $\eqref{eq:Bconv_bound}$ and $\eqref{eq:Qconv_bound}$ and Lemma \ref{lem:Fterm}, for every $T>0$ there exists some function $r_T(\epsilon)$ going to $0$, as $\epsilon \to 0$, such that
\begin{align*}
\mathbb{E} \sup_{\delta \leq t \leq T} |u_\epsilon^x(t) - u^{x}(t)|_{H_\mu} \leq c\,e^{-\frac{\gamma \delta}\epsilon} |x|_{H_\mu} + r_T(\epsilon) + c_T\,\int_\delta^T \mathbb{E}\sup_{\delta \leq s \leq t}|u_\epsilon^x(s) - u^x(s)|_{H_\mu}dt.
\end{align*}
Then, using Gronwall's Lemma, we have
\begin{equation}\label{eq:lem22_2}
\mathbb{E} \sup_{\delta \leq t \leq T}|u_\epsilon^x - u^x|_{H_\mu} \leq c\,\le(e^{-\frac{\gamma \delta}\epsilon}|x|_{H_\mu} + r_T(\epsilon)\r)e^{c_T T}.
\end{equation}
Meanwhile, we can estimate the second term in $\eqref{eq:lem22_1}$ by using the bounds $\eqref{eq:Bconv_bound}$, $\eqref{eq:Qconv_bound}$ and $\eqref{eq:apriori}$ to obtain
\begin{equation}
\label{eq:lem22_3}
\begin{array}{l}
\ds{
 \mathbb{E} \sup_{0 \leq t \leq \delta} |u_\epsilon^x(t)  - e^{t \frac A\e}x |_{H_\mu} \leq c\, \sqrt{\gamma(\epsilon)} + \mathbb{E}\sup_{0 \leq t \leq \delta} \le|\int_0^t e^{(t-s)\frac A\e}F(s,u_\epsilon^x(s))ds \r|_{H_\mu} }\\
\vspace{.1mm}\\
\ds{  \leq c\,\le( \sqrt{\gamma(\epsilon)}+\delta \le(1+ \mathbb{E}|u_\epsilon^x|_{C([0,\delta];H_\mu)}\r)   \r) \leq c\,(\sqrt{\gamma(\epsilon)} + \delta).}
\end{array}
\end{equation}
This, together with \eqref{eq:lem22_1} and \eqref{eq:lem22_2}, implies that for every $\d\in\,(0,T_x)$
\[\mathbb{P}\le(u_\epsilon^x(\sigma_\epsilon^x(\rho)) \in \partial D\r)\leq c_T\,\le(\sqrt{\gamma(\e)}+\d+r_T(\e)+e^{-\frac {\gamma \d}{\e}}\,|x|_{H_\mu}\r).\]
Thus, by taking $\delta = \epsilon^r$ for some $ 0 < r < 1$, we get
\[\mathbb{P}\le(u_\epsilon^x(\sigma_\epsilon^x(\rho)) \in \partial D\r)=0.\]

\end{proof}

\begin{proof}[Proof of Lemma \ref{lem:23}]
We have
\begin{equation*}
u_\epsilon^x(t) - x = e^{t \frac A\e} x - x + \int_0^t e^{(t-s)\frac A\e} F(s,u^x_\epsilon(s))ds + \a(\epsilon)\, w_{A,Q}^\epsilon(u^x_\epsilon)(t) + \beta(\epsilon)\, w_{A,B}^\epsilon(t).
\end{equation*}
Since the semigroup $e^{tA}$ acts as a contraction on $H_\mu$, we have that $|e^{t \frac A\e}x-x|_{H_\mu} \leq 2 |x|_{H_\mu}$. Next we observe that, for $t \in [0,T]$,
\[\begin{array}{l}
\ds{\le|\int_0^t e^{(t-s)\frac A\e} F(s,u^x_\epsilon(s))ds \r|_{H_\mu}  \leq  \int_0^t |F(s,u_\epsilon^x(s))|_{H_\mu}ds }\\
\vspace{.1mm}\\
\ds{ \leq c\, T\le(1+ \sup_{0 \leq s \leq t}|u^x_\epsilon(s)|_{H_\mu}\r) \leq c\,T \le(1 + |x|_{H_\mu}+ \sup_{0 \leq s \leq T} |u^x_\epsilon(s)-x |_{H_\mu} \r).}
\end{array}\]
Therefore, if $x \in B_\rho$, we can find a $T_\rho>0$ small enough that
\begin{align*}
\sup_{0 \leq t \leq T} |u_\epsilon^x(t)-x|_{H_\mu} \leq \frac{7\rho}{3} + \a(\e)\,\sup_{0 \leq t \leq T_\rho}|w_{A,Q}^\epsilon(u^x_\epsilon)(t)|_{H_\mu} +  \beta(\epsilon)\,\sup_{0 \leq t \leq T_\rho}|w_{A,B}^\epsilon(t)|_{H_\mu}.
\end{align*}
Hence, 
\[\begin{array}{l}
\ds{\mathbb{P}\le(|u_\epsilon^x - x |_{C([0,T];H_\mu)} \geq 3\rho\r)}\\
\vspace{.1mm}\\
\ds{\leq
\mathbb{P}\le(\a(\epsilon)\, |w_{A,Q}^\epsilon(u_\epsilon^x)|_{C([0,T];H_\mu)} \geq \rho/3\r)+\mathbb{P}\le(\beta(\epsilon)\,|w_{A,B}^\epsilon|_{C([0,T];H_\mu)} \geq \rho/3\r ).}
\end{array}\]
Thanks to Condition (i) of Hypothesis \ref{hyp:exit}, the integrand of $w^{A,Q}(u_\epsilon^x)$ is bounded, so that we can use the exponential estimates for the stochastic convolution (see \cite{peszat}). In particular, for every $T>0$ we have
\begin{equation*}
\mathbb{P} \le( \sup_{0 \leq t \leq T} |w_{A,Q}^\epsilon(u_\epsilon^x) (t)|_{H_\mu} \geq \frac{\rho}{3\,\a(\e)} \r) \leq c\exp\le(-\frac{\rho^2}{3c_T\,\a(\e)} \r)\leq c\,\exp\le(-\frac{\rho^2}{3\,c_t\,\gamma(\e)}\r),
\end{equation*}
where $c_T$ is a constant going to $0$, as $T \to 0$. We obtain a similar estimate for $w_{A,B}^\epsilon$, with $\a(\e)$ replaced by $\beta(\e)$. All together, for every $T\leq T_\rho$ we have
\begin{equation*}
\gamma(\epsilon) \log \sup_{x \in B_\rho} \mathbb{P}(\sup_{0 \leq t \leq T} |u_\epsilon^x(t) - x|_{H_\mu}\geq 3\rho) \leq c \,\gamma(\epsilon) - \frac{\rho^2}{c_T}.
\end{equation*}
Upon taking $T$ small enough, this gives us the desired result.
\end{proof}

\begin{proof}[Proof of Lemma \ref{lem:21}]
We modify the proof of Lemma 5.7.21 in \cite{dembo2009large}  to account for the behavior of $u_\epsilon^x(t)$ near $t = 0$. Let $N \subset \partial D$ be a closed set. Define the closed set 
\[\Psi_{\delta,T}(N):= \{u \in C([0,T];H_\mu): \exists \ t \in [\delta,T] \text{ such that } u(t) \in N \}.\]
 Then, for any $T > 0$ and $\delta < T$,
\begin{equation}\label{eq:lem21_1}
\mathbb{P}(u_\epsilon^x(\sigma_\epsilon^x(\rho)) \in N) \leq \mathbb{P}(\tau_\epsilon^x< \delta)+ \mathbb{P}(\sigma_\epsilon^x(\rho) > T ) + \mathbb{P}(u_\epsilon^x \in \Psi_{\delta,T}(N)).
\end{equation}
To bound the first term from above, we notice that
\begin{align*}
\sup_{x \in \partial B_{2 \rho}} \mathbb{P}(\tau_x^\epsilon < \delta) \leq \sup_{x \in \partial B_{2 \rho}} \mathbb{P}\le(\sup_{0 \leq t \leq \delta} |u_\epsilon^x(t)-x|_{H_\mu} \geq \mathrm{dist}_{H_\mu}(x,\partial D)\r).
\end{align*}
Now, let $\rho>0$ be small enough that $\inf_{x \in \partial B_{2\rho}} \mathrm{dist}(x,\partial D) \geq 6 \rho$. Then, by Lemma \ref{lem:23}, the inequality above implies that for any $\eta > 0$ there exists  $\delta > 0$ small enough that
\begin{equation}\label{eq:lem21_2}
\limsup_{\epsilon \to 0} \gamma(\epsilon) \log \sup_{x \in \partial B_{2 \rho}} \mathbb{P}(\tau_x^\epsilon < \delta) \leq - \eta.
\end{equation}
Next, thanks to Lemma \ref{lem:19}, we can find $T>0$ large enough that
\begin{equation}\label{eq:lem21_3}
\limsup_{\epsilon \to 0} \gamma(\epsilon) \log \sup_{x \in \partial B_{2\rho}} \mathbb{P}(\sigma_\epsilon^x(\rho)> T) < -\eta.
\end{equation} 
Since the set $\Psi_{\delta,T}(N)$ is closed, we can use the large deviation principle and equation $\eqref{eq:LDP_upper}$ to obtain that
\begin{equation}\label{eq:lem21_4}
\limsup_{\epsilon \to 0} \gamma(\epsilon) \log \sup_{x \in \partial B_{2\rho}} \mathbb{P}(u_\epsilon^x \in \Psi_{\delta,T}(N)) \leq - \inf_{x \in \partial B_{2\rho}} I_{\delta,T}^x(\Psi_{\delta,T}(N)).
\end{equation}
On the other hand, for fixed $x$, we have that
\begin{equation*}
I_{\delta,T}^x(\Psi_{\delta,T}(N)) = \inf_{\varphi \in \Psi_{\delta,T}(N)} I^x_{\delta,T}(\varphi) \geq \inf_{\varphi \in \Psi_{0,T}(N)} I^x_{0,T}(\varphi),
\end{equation*}
because every path hitting $N$ in the interval $[\delta,T]$ has an extension to a path on $[0,T]$ starting at $x$. 

Next, we notice that for any $x \in \partial B_{2\rho}$,
\begin{equation*}
V(x) + \inf_{ \varphi \in \Psi_{0,T}(N)} I^x_{0,T}(\varphi) \geq \inf_{z \in N} V(z),
\end{equation*}
since any path on the left hand side is also considered in the infima on the right hand side. Now, due to Hypotheses \ref{hyp:exit}, it is clear that $\lim_{x \to 0} V(x) = 0$. Hence, for any $\gamma > 0$, if we choose $\rho>0$ small enough then, thanks to $\eqref{eq:lem21_4}$, we have
\begin{equation}\label{eq:lem21_5}
\limsup_{\epsilon \to 0} \gamma(\epsilon) \log \sup_{x \in \partial B_{2\rho}} \mathbb{P}(u_\epsilon^x \in \Psi_{\delta,T}(N)) \leq \gamma -  \inf_{z \in N} V(z).
\end{equation}
Due to $\eqref{eq:lem21_1}$, $\eqref{eq:lem21_2}$, $\eqref{eq:lem21_3}$, $\eqref{eq:lem21_5}$ and the arbitrariness of $\gamma$, the result then follows by picking $\eta > \inf_{z \in N} V(z)$.
\end{proof}



\begin{thebibliography}{99}


\bibitem{bcf} Z.~Brze{\'z}niak, S.~Cerrai, M.~Freidlin, {\em Quasipotential and exit times for 2D Stochastic Navier-Stokes equations driven by   space-time white noise},   Probability Theory and Related Fields 162 (2015), pp. 739--793.



\bibitem{BM2} D.~Bl\^omker, W.~W.~ Mohammed,  {\em Fast diffusion limit for reaction-diffusion systems with stochastic Neumann boundary conditions}, SIAM Journal on Mathematical Analysis 48 (2016), pp. 3547--3578.
 

\bibitem{budhiraja2008}
A.~Budhiraja, P.~Dupuis, V.~Maroulas, {\em Large deviations for infinite dimensional stochastic dynamical systems}, The Annals of Probability 36 (2008), pp. 1390--1420.


\bibitem{budhiraja2001} A.~Budhiraja, P.~Dupuis, Paul, {\em  Variational Representation for Positive Functionals of Infinite Dimensional Brownian Motion}, Probab. Math. Statist. 20 (2011).


\bibitem{cf} S.~Cerrai, M.~Freidlin, {\em Approximation of quasi-potentials and exit problems for multidimensional RDE's with noise}, Transactions of the AMS 363 (2011), pp. 3853-3892.


\bibitem{cerrai2011}
S.~Cerrai, M.~Freidlin, {\em Fast transport asymptotics for stochastic RDEs with boundary noise}, The Annals of Probability, 39 (2011), pp. 369--405.

\bibitem{cs} S.~Cerrai, M.~Salins, {\em Smoluchowski-Kramers approximation and  large deviations for  infinite dimensional non-gradient systems with  applications to the exit problem},   The Annals of Probability 44 (2016), pp. 2591--2642.


\bibitem{dpz-93} G.~Da Prato, J.~Zabczyk, {\em Evolution equations with white-noise boundary conditions},
Stochastics and Stochastics Reports 42 (1993), pp. 167--182. 



\bibitem{dapratozabczyk1996} G.~Da Prato, J.~Zabczyk, {\sc Ergodicity for Infinite Dimensional Systems}, London Mathematical Society Lecture Note Series, Cambridge University Press, 1996.



\bibitem{davies1990heat} E.~B.~Davies,
{\sc Heat Kernels and Spectral Theory}
Cambridge Tracts in Mathematics, Cambridge University Press,1990.

\bibitem{dembo2009large} A.~Dembo, A. O.~Zeitouni, {\sc Large Deviations Techniques and Applications}, Springer Verlag, 2009.

\bibitem{Dupuis:2104622}
P.~Dupuis, R.~Ellis, {\sc Weak Convergence Approach to the Theory of Large Deviations}, Wiley series in probability and statistics
Wiley, 2011.

\bibitem{fw-92} M.~Freidlin, A.~D.~Wentzell, {\em Reaction-diffusion equations with randomly
perturbed boundary conditions}, Annals of Probability 20 (1992), pp. 963--986. 

\bibitem{freidlin1998random} M.~I.~ Freidlin, A.~D.~Wentzell, {\sc Random Perturbations of Dynamical Systems}, third edition, Springer Verlag, 2012.

\bibitem{Lasiecka1980}
I.~Lasiecka, {\em Unified theory for abstract parabolic boundary problems---a semigroup approach}, 
Applied Mathematics and Optimization 6, 
(1980).




\bibitem{lions1972non} J.~L.~Lions, E.~Magenes, {\sc Non-homogeneous Boundary Value Problems and Applications}, Springer-Verlag, 1972.









\bibitem{peszat} S.~Peszat, {\em Exponential tail estimates for infinite-dimensional stochastic convolutions}, Bulletin of the Polish Academy of Sciences, Mathematics 40  (1992), pp. 323--333.



\bibitem {salins:171207231S}
M.~Salins, {\em Equivalences and counterexamples between several definitions of the uniform large deviations principle,} 
ArXiv: 1712.07231 (2017).

\bibitem{SV} R.~Schnaubelt, M.~Veraar, {\em Stochastic Equations with Boundary Noise}, 
Progr. Nonlinear Differential Equations Appl., 80 (2011), pp. 609--629.


\bibitem{sowers} R.~B.~Sowers, {\em Multidimensional reaction-diffusion equations with white noise boundary perturbations}
Annals of Probabability, 22 (1994), pp. 2071--2121.





\end{thebibliography}
\end{document}